\documentclass[reqno,oneside]{amsart}
\usepackage{amssymb,bm}
\usepackage{amsmath,mathtools}
\usepackage{marginnote}
\usepackage{algorithm}
\usepackage{algpseudocode}

\usepackage{tikz}
\usetikzlibrary{cd}

\usepackage[colorlinks=true]{hyperref}
\hypersetup{urlcolor=blue, citecolor=red}

\newtheorem{Theorem}{Theorem}[section]
\newtheorem{Corollary}[Theorem]{Corollary}
\newtheorem{Lemma}[Theorem]{Lemma}
\newtheorem{Proposition}[Theorem]{Proposition}

\theoremstyle{definition}
\newtheorem{Remark}[Theorem]{Remark}
\newtheorem{Assumption}[Theorem]{Assumption}

\theoremstyle{definition}
\newtheorem{Experiment}{Experiment}

\numberwithin{equation}{section}

\makeatletter
\def\mathcenterto#1#2{\mathclap{\phantom{#1}\mathclap{#2}}\phantom{#1}}
\let\old@widetilde\widetilde
\def\widetildeto#1#2{\mathcenterto{#2}{\old@widetilde{\mathcenterto{#1}{#2}}}}
\let\old@widehat\widehat
\def\widehatto#1#2{\mathcenterto{#2}{\old@widehat{\mathcenterto{#1}{#2\,}}}}
\makeatother

\def\widetilde{\widetildeto{K}}

\def\R {\mathbb R}

\newcommand{\<}{\langle}
\renewcommand{\>}{\rangle}
\newcommand{\Id}{\operatorname{Id}}

\newcommand{\inclusion}{\hookrightarrow}

\newcommand{\p}{\partial}
\newcommand{\Vol}{\operatorname{Vol}}

\newcommand{\diam}{\operatorname{diam}}
\newcommand{\dist}{\operatorname{dist}}

\renewcommand{\span}{\operatorname{span}}

\newcommand{\cT}{\mathcal T}

\newcommand{\cW}{\mathcal W}

\begin{document}
\title[The geodesic ray transform with finite measurements]{Inverting the geodesic ray transform with finite measurements: stability and reconstruction}

\author[Yernat M. Assylbekov]{Yernat M. Assylbekov}
\address{Super Dispatch}
\email{yernat.assylbekov@gmail.com}

\begin{abstract}
We study the inversion of the geodesic ray transform from finitely many local averages of its data. Under suitable geometric assumptions, we prove Lipschitz stability on any fixed finite-dimensional reconstruction space when the measurement partition is sufficiently fine. We develop convergent reconstruction algorithms based on Steepest Gradient Descent and Conjugate Gradients and implement them using piecewise constant finite element spaces. Numerical experiments in 2D and 3D illustrate the influence of geometry and measurement discretization on reconstruction quality.
\end{abstract}

\maketitle

\section{Introduction}
\subsection{Statement of the problem}
Consider a smooth compact $n$-dimensional Riemannian manifold $(M,g)$ with boundary $\p M$. Let $SM$ be its unit sphere bundle. We assume that $\p M$ is strictly convex and $(M,g)$ is \emph{non-trapping}. The latter means that for any $(x, v)\in SM$, the first non-negative time $\tau(x, v)$ when the unique geodesic $\gamma_{x,v}$ with $x=\gamma_{x,v}(0)$ and $v=\dot\gamma_{x,v}(0)$ exits $M$ is finite.

Let $\p_\pm SM$ be the set of inward/outward unit vectors on $\p M$, i.e.
$$
\p_\pm SM:=\{(x,v)\in SM:x\in\p M\text{ and }\pm\<v,\nu(x)\>_{g(x)}\ge 0\},
$$
where $\nu$ is the inward unit normal to $\p M$. The object of
study is the \emph{geodesic ray transform} defined for $f\in C(M)$ as
$$
I_g f(x,v):=\int_0^{\tau(x,v)} f(\gamma_{x,v}(t))\,dt,\quad (x,v)\in\p_+ SM.
$$
The best-known example is the classical Radon transform, introduced in \cite{Ra1917}. In two dimensions with the Euclidean metric, geodesics are straight lines, and $I_g$ agrees, up to parametrization, with the Radon transform of a function extended by zero outside $M$. This transform provides the mathematical foundation of X-ray computed tomography (CT). Related ray transforms, including weighted and attenuated variants, also arise in positron emission tomography (PET) and single-photon emission computed tomography (SPECT) \cite{KaSl2001SIAM}. The geodesic ray transform extends straight-line integration to curved trajectories determined by the metric $g$.

The domain of $I_g$ can be extended to $L^2(M)$; see Section~\ref{section: preliminaries} for details. The classical inverse problem is to reconstruct $f\in L^2(M)$ from complete knowledge of $I_g f$. In practical applications, however, the unknown function is approximated in a finite-dimensional space, and only finitely many measurements are available. We consider a finite-resolution measurement model consisting of local averages of $I_g f$ over a partition of $\p_+ SM$. For a prescribed finite-dimensional reconstruction space, we ask whether sufficiently fine measurement partitions preserve stability. The averaging model and its motivation from finite-aperture tomography are discussed in Section~\ref{section: stability estimates}.

\subsection{Previous work}
We first recall results for the classical inverse problem, where the complete geodesic ray transform $I_gf$ is assumed to be known, corresponding to an infinite number of measurements.

The injectivity of $I_g$ was first established for simple surfaces in \cite{Mu1977DANSSSR}. The proof, based on energy methods parameterizing geodesics by their endpoints, also yields an $L^2 \to H^1$ stability estimate. This injectivity result was subsequently extended to arbitrary dimensions $n \ge 2$ in \cite{BeGe1978DANSSSR, Mu1981SMJ}. Obtaining stability estimates is more delicate, mainly because of the behavior near the boundary. Under suitable curvature assumptions, an $L^2 \to H^1$ stability estimate for $n \ge 2$ was proved in \cite{Sh1999NSC}. This result was refined in \cite{PaSa2021}, where an $L^2 \to H^{1/2}$ stability estimate was derived under a different curvature condition. Both works parameterize geodesics by $\p_+ SM$. For some earlier related works, see \cite{Na1986BGT, Sh1994VSP}.

A different line of research is based on microlocal analysis. Using scattering calculus, \cite{UhVa2016IM} established injectivity together with an $L^2 \to H^1$ stability estimate for $n \ge 3$ under a global strict convex foliation assumption, even in the presence of conjugate points. For the $\p_+ SM$ parameterization, \cite{HoUh2018JDG} obtained a sharp stability estimate of type $H^{-1/2}\to L^2$. Another sharp estimate, of type $L^2\to H^{1/2}$, was established in \cite{AsSt2020IP} by applying microlocal techniques to alternative parameterizations of geodesics. See also \cite{Mo2020SIAMJMA} for a related result.

A couple of reconstruction methods have been proposed. The scattering calculus approach of \cite{UhVa2016IM} also provides a reconstruction algorithm for inverting $I_g$, whose numerical implementation was presented in \cite{AuChUh2019SIAMJIS}. Earlier, explicit inversion formulas for simple surfaces with curvature close to constant were derived in \cite{Kr2010JIIPP, PeUh2004IMRN}, and their numerical implementation was demonstrated in \cite{Mo2014SIAMJIS}.

A finite number of measurements has also been studied from the viewpoint of sampling and resolution. Sampling and aliasing are studied for the geodesic X-ray transform on simple surfaces in~\cite{MoSt2023SIAMJMA} and for the Radon transform with finitely many angles in~\cite{St2023IP}. Local resolution near singularities for discrete Radon and generalized Radon data is analyzed in~\cite{Ka2017SIAMJAM, Ka2023JFAA}, including iterative reconstruction in~\cite{Ka2025SIAMJIS}.

\subsection{Approach and main results}
Our approach specializes the finite-measurement framework of \cite{AlSa2022ARMA} to $I_g$. On a fixed finite-dimensional space $\cW\subset L^2(M)$, injectivity gives full-data $L^2 \to L^2$ type Lipschitz stability by compactness. Strong convergence of the local-average projections $P_\rho^{\p_+ SM}$ then yields operator-norm convergence on $I_g(\cW)$, preserving stability for sufficiently fine measurement meshes.

Let $\rho$ denote the measurement mesh size, an upper bound on the diameters of its cells as defined in Section~\ref{section: piecewise constant FES on influx}. Under the geometric hypotheses of Assumption~\ref{assumption: simplicity and strict convex admittance of M}, Corollary~\ref{corollary: main stability} gives a threshold $\rho_0(g,\cW)>0$ such that every measurement mesh with $\rho\le\rho_0$ satisfies

$$
    \|f\|_{L^2(M)} \le \frac{2}{\sqrt{3}}C_{g,\cW} \|P_\rho^{\p_+ SM}I_g(f)\|_{L^2_\mu(\p_+ SM)}, \qquad f\in\cW,
$$
where $C_{g,\cW}$ is a full-data stability constant.

For piecewise constant reconstruction spaces, this stability ensures the convergence of steepest gradient descent and conjugate gradients from any initial guess in this space. Both admit geometric error bounds, and conjugate gradients terminate within $\dim\cW$ iterations (Theorems~\ref{theorem: steepest gd} and~\ref{theorem: cgm}). These methods exploit the linear least-squares structure within the stability-to-reconstruction strategy of~\cite{AlSa2022ARMA}, where the Landweber iteration is used. Experiments in 2D and 3D illustrate the effects of geometry and measurement discretization.

\subsection{Structure of the paper}
The paper is organized as follows. In Section~\ref{section: preliminaries}, we collect some basic facts about the geodesic ray transform $I_g$. In Section~\ref{section: piecewise constant FES on influx}, we construct the orthogonal projection $P^{\p_+ SM}_\rho$ onto a piecewise constant finite element space on $\p_+ SM$. Applied to $I_g f$, this projection represents the local-average measurements as a piecewise constant approximation of the full data.

In Section~\ref{section: stability estimates}, we establish $L^2\to L^2$ type stability estimates using functional analytic techniques. In Section~\ref{section: reconstructions}, we derive globally convergent reconstruction algorithms using the stability estimates established in the previous section.

In Section~\ref{section: experiments}, we present numerical experiments for the implementation of the reconstruction algorithms constructed in Section~\ref{section: reconstructions}. Finally, in Section~\ref{section: conclusion}, we provide concluding remarks and discuss directions for future research.\medskip

\noindent{\bf Acknowledgements.} The author is grateful to Professor François Monard for reading an earlier version of the manuscript, offering very useful suggestions, and drawing the author's attention to the papers~\cite{Ka2017SIAMJAM, Ka2023JFAA, Ka2025SIAMJIS}. The author acknowledges substantial help from ChatGPT and OpenAI Codex, using the GPT-5.6 Sol and GPT-6 Astra models, with numerical approaches, proof sketches, Python programming, literature-search suggestions, drafting, and editing. The reconstruction approach originally used Landweber iteration; AI suggested considering steepest gradient descent and the conjugate gradient method in Section 5 and offered initial ideas for the proofs of Theorems 5.1 and 5.2. The author revised, simplified, and completed these arguments. AI assistance also helped improve the writing throughout the paper, especially in Section 6. The author formulated the inverse problem, developed the stability estimates and the Landweber reconstruction approach, and chose the final arguments. The author also substantially simplified and revised the generated code, ran and checked the computations, verified the references, and reviewed the manuscript. Before finalizing the paper, the author personally reviewed and checked the entire manuscript, including all writing, mathematical statements, assumptions, derivations, and proofs, as well as the overall article structure, notation, algorithms, numerical results, figures, references, and conclusions. This comprehensive review covered every part of the final version, including all material developed with AI assistance. The author takes responsibility for the paper’s mathematical claims, computational results, references, and statements of originality.

\section{Preliminaries}\label{section: preliminaries}
\subsection{Functional Spaces}
Let $(M,g)$ be a compact Riemannian manifold with boundary. We denote by $d\Vol_g$ the Riemannian volume form on $M$, by $d\Sigma^{2n-1}$ the Liouville volume form on $SM$, and by $d\Sigma^{2n-2}$ its induced volume form on $\p_+ SM$. We define the measure $d\mu$ on $\p_+ SM$ by
$$
d\mu(x,v) := \< v,\nu(x)\>_{g(x)}\,d\Sigma^{2n-2}(x,v),
\qquad (x,v)\in\p_+ SM,
$$
where $\nu$ is the inward unit normal to $\p M$. We use $L^2(M)$, $L^2(SM)$, and $L^2_\mu(\p_+ SM)$ to denote the completions of $C(M)$, $C(SM)$, and $C(\p_+ SM)$, respectively, with respect to the corresponding $L^2$-norms induced by $d\Vol_g$, $d\Sigma^{2n-1}$, and $d\mu$.

\subsection{The geodesic ray transform}
Now assume that $(M, g)$ is non-trapping and $\p M$ is strictly convex. Throughout this paper, we define
$$
    \diam(M):=\sup_{(x, v)\in \p_+SM}\tau(x,v) < \infty.
$$
For $f\in C(M)$, the Cauchy-Schwarz inequality along each geodesic and Santal\'o's formula \cite{Sh1994VSP} yield
$$
    \|I_gf\|_{L^2_\mu(\p_+SM)}^2
    \le \diam(M)\int_{SM}|f(x)|^2\,d\Sigma^{2n-1}(x,v)
    =\diam(M)\,\omega_{n-1}\|f\|_{L^2(M)}^2.
$$
Here we used $d\Sigma^{2n-1}(x,v)=d\Vol_g(x)\,d\sigma_x(v)$, where $d\sigma_x$ is the surface measure on $S_x M$ whose total measure is $\omega_{n-1}$, the area of the Euclidean unit sphere in $\R^n$. Therefore, $I_g$ extends uniquely to a bounded linear operator $I_g:L^2(M)\to L^2_\mu(\p_+ SM)$ satisfying
\begin{equation}\label{equation: explicit bound for operator norm for I_g acting on L^2 functions}
    \|I_g\|_{L^2(M)\to L^2_\mu(\p_+ SM)} \le \sqrt{\diam(M)\,\omega_{n-1}}.
\end{equation}
Now consider the adjoint operator $I_g^*: L^2_\mu(\p_+SM)\to L^2(M)$ that is characterized by
\begin{equation}\label{equation: adjoint identity for I_g}
    (I_gf,h)_{L^2_\mu(\p_+SM)}=(f,I_g^*h)_{L^2(M)}
\end{equation}
for which we now give an expression. Given $h\in L^2_\mu(\p_+ SM)$, extend it as a constant along each geodesic by
$$
    h_\psi(x, v) := h\big(\gamma_{x, v}(-\tau(x, -v)), \dot\gamma_{x, v}(-\tau(x, -v))\big),
    \qquad (x, v)\in SM.
$$
Then by Santal\'o's formula
$$
    \|h_\psi\|_{L^2(SM)}^2 = \int_{\p_+SM}\tau(x,v)|h(x,v)|^2\,d\mu(x,v) \le \diam(M)\|h\|_{L^2_\mu(\p_+ SM)}^2,
$$
which makes the extension well defined on $L^2$ equivalence classes. Applying the same Santal\'o's formula to \eqref{equation: adjoint identity for I_g} gives
\begin{equation}\label{equation: explicit adjoint of I_g}
    I_g^*h(x) = \int_{S_x M}h_\psi(x, v)\,d\sigma_x(v),\qquad x\in M.
\end{equation}

\section{Piecewise Constant Finite Element Subspaces on $\p_+ SM$}\label{section: piecewise constant FES on influx}
Let $(M, g)$ be a compact Riemannian manifold with boundary. A \emph{mesh} $\cT_\rho(\p_+ SM)$ of \emph{size} $\rho>0$ on $\p_+ SM$ is a finite collection of closed subsets of $\p_+ SM$, called \emph{cells}, with non-empty interiors, such that
$$
    \begin{gathered}
        \p_+ SM = \bigcup_{E\in \cT_\rho(\p_+ SM)} E, \qquad \mu(E \cap F) = 0 \quad \text{whenever} \quad E, F \in \cT_\rho(\p_+ SM),\, \quad E \neq F,\\
        \max_{E\in \cT_\rho(\p_+ SM)}\diam_{\p_+ SM}(E) \le \rho,
    \end{gathered}
$$
where $\diam_{\p_+ SM}$ denotes the diameter with respect to the metric induced on $\p_+ SM$.

Now consider the piecewise constant finite element subspace (FES) $V_\rho(\p_+ SM)\subset L^2_\mu(\p_+ SM)$ induced by $\cT_\rho(\p_+ SM)$:
$$
    V_\rho(\p_+ SM) := \span\{\psi_E\, :\,  E\in \cT_\rho(\p_+ SM)\},
$$
where
\begin{equation}\label{equation: orthonormal basis of FES}
    \psi_E := \frac{1}{\mu(E)^{1/2}}\mathbf{1}_{E}, \quad E \in \cT_\rho(\p_+ SM)
\end{equation}
are orthonormal basis functions. Note that the subspace $V_\rho(\p_+ SM)\subset L^\infty(\p_+ SM)$ is finite-dimensional and essentially the span of the characteristic functions of sets in $\cT_\rho(\p_+ SM)$.

\begin{Remark}
    Since the influx measure $d\mu$ degenerates at $S(\p M)$, the basis functions induced by a mesh $\cT_\rho(\p_+ SM)$ may become large on cells approaching $S(\p M)$. To prevent such behavior, one can impose a shape-regularity condition on the mesh. We say that $\cT_\rho(\p_+ SM)$ is \emph{shape-regular} if there exist constants $C, c > 0$, independent of cells, such that
    $$
        c\, \diam_{\p_+ SM}(E)^{2n-2} \le \mu(E) \le C\, \diam_{\p_+ SM}(E)^{2n-2}, \qquad E\in \cT_\rho(\p_+ SM).
    $$
    This condition can be used for numerical purposes to control the values of $\psi_E$ by choosing $E$ with larger $\diam_{\p_+ SM}(E)$ near $S(\p M)$. This highlights the need for an adaptive mesh near $S(\p M)$ for numerical implementations. The design of efficient meshes is outside the scope of the paper. We refer the interested reader to the specialized literature on this topic, such as \cite{Ed2001CUP, ThWaMa1985NH}.
    
\end{Remark}

Consider the projection $P_\rho^{\p_+ SM}: L^2_\mu(\p_+ SM) \to V_\rho(\p_+ SM)$ defined as
$$
    (f - P_\rho^{\p_+ SM} f, v)_{L^2_\mu(\p_+ SM)} = 0 \quad \text{for all} \quad v\in V_\rho(\p_+ SM),
$$
which can be expressed as follows
\begin{equation}\label{equation: P_rho operator expression}
    \begin{aligned}
        P_\rho^{\p_+ SM} f &= \sum_{E\in \cT_\rho(\p_+ SM)} (f, \psi_E)_{L^2_\mu(\p_+ SM)} \psi_E \\
        &= \sum_{E\in \cT_\rho(\p_+ SM)} \frac{1}{\mu(E)} \bigg(\int_{E}f(y, w)\,d\mu(y, w)\bigg)\, \mathbf{1}_{E},\qquad f\in L^2_\mu(\p_+ SM).
    \end{aligned}
\end{equation}
Let $\imath_\rho^{\p_+ SM}$ denote the inclusion operator $V_\rho(\p_+ SM) \inclusion L^2_\mu(\p_+ SM)$.

\begin{Proposition}\label{proposition: projection properties}
The operator $P_\rho^{\p_+ SM}$ satisfies
\begin{enumerate}
    \item[(a)] $(P_\rho^{\p_+ SM})^* = \imath_\rho^{\p_+ SM}$ and $(\imath_\rho^{\p_+ SM})^* = P_\rho^{\p_+ SM}$.

    \item[(b)] $\|P_\rho^{\p_+ SM}\|_{L^2_\mu(\p_+ SM) \to L^2_\mu(\p_+ SM)} = 1$ and $\|\Id - P_\rho^{\p_+ SM}\|_{L^2_\mu(\p_+ SM) \to L^2_\mu(\p_+ SM)} = 1$.

    \item[(c)] If $f\in L^2_\mu(\p_+ SM)$, then
    $$
        \|f - P_\rho^{\p_+ SM} f\|_{L^2_\mu(\p_+ SM)} = \inf_{v \in V_\rho(\p_+ SM)} \|f - v\|_{L^2_\mu(\p_+ SM)}.
    $$
\end{enumerate}
\end{Proposition}

\begin{proof}
    For $f\in L^2_\mu(\p_+SM)$ and $v\in V_\rho(\p_+SM)$, using \eqref{equation: P_rho operator expression} one can directly show
    $$
        (P_\rho^{\p_+ SM} f, v)_{L^2_\mu(\p_+ SM)} = (f, \imath_\rho^{\p_+ SM} v)_{L^2_\mu(\p_+ SM)}.
    $$
    Thus $(P_\rho^{\p_+SM})^*=\imath_\rho^{\p_+SM}$ and $(\imath_\rho^{\p_+ SM})^* = P_\rho^{\p_+ SM}$. Next, for $f\in L^2_\mu(\p_+ SM)$, orthogonality gives
    $$
        \begin{aligned}
            \|f\|_{L^2_\mu(\p_+ SM)}^2 &= \|P_\rho^{\p_+ SM} f + (f - P_\rho^{\p_+ SM} f)\|_{L^2_\mu(\p_+ SM)}^2 \\
            &= \|P_\rho^{\p_+ SM} f\|_{L^2_\mu(\p_+ SM)}^2 + \|f - P_\rho^{\p_+ SM} f\|_{L^2_\mu(\p_+ SM)}^2.
        \end{aligned}
    $$
    Hence, $\|P_\rho^{\p_+ SM}\|_{L^2_\mu(\p_+ SM) \to L^2_\mu(\p_+ SM)} \le 1$ and $\|\Id - P_\rho^{\p_+ SM}\|_{L^2_\mu(\p_+ SM) \to L^2_\mu(\p_+ SM)} \le 1$. Each operator acts as the identity on its nonzero range, so both norms equal one. Finally, for every $v\in V_\rho(\p_+SM)$, orthogonality gives
    $$
        \|f - v\|_{L^2_\mu(\p_+ SM)}^2 = \|f - P_\rho^{\p_+ SM}f\|_{L^2_\mu(\p_+ SM)}^2 + \|P_\rho^{\p_+ SM}f - v\|_{L^2_\mu(\p_+ SM)}^2.
    $$
    The minimum is therefore attained at $v=P_\rho^{\p_+SM}f$.
\end{proof}

We conclude this section with two approximation results for the projection $P_\rho^{\p_+ SM}$.

\begin{Lemma}\label{lemma: approximations of C functions with P_rho}
If $f\in C(\p_+ SM)$, for every $\varepsilon > 0$, there exists $\rho_0 > 0$ such that the induced projection $P_\rho^{\p_+ SM}: L^2_\mu(\p_+ SM) \to V_\rho(\p_+ SM)$ of any mesh $\cT_\rho(\p_+ SM)$ on $\p_+ SM$ with $\rho \le \rho_0$ satisfies $\|f - P_\rho^{\p_+ SM} f\|_{L^\infty(\p_+ SM)} < \varepsilon$.
\end{Lemma}

\begin{proof}
The compactness of $\p_+ SM$ implies the uniform continuity of $f$. In other words, there exists $\rho_0 > 0$ such that $|f(x, v) - f(y, w)| < \varepsilon$ if $\dist_{\p_+ SM}\big((x, v), (y, w)\big) \le \rho_0$, where $\dist_{\p_+ SM}$ denotes the distance function with respect to the metric induced on $\p_+ SM$. Thus, for any mesh $\cT_\rho(\p_+ SM)$ on $\p_+ SM$ with $\rho \le \rho_0$,
\begin{equation}\label{equation: uniform continuity on E}
    |f(x, v) - f(y, w)| < \varepsilon \quad \text{for} \quad (x, v), (y, w)\in E, \quad E\in \cT_\rho(\p_+ SM).
\end{equation}
By \eqref{equation: P_rho operator expression}, we have
$$
    |f(x, v) - P_\rho^{\p_+ SM} f(x, v)| \le \frac{1}{\mu(E)} \int_{E}|f(x, v) - f(y, w)|\,d\mu(y, w) \quad \text{for} \quad (x, v)\in E, \quad E\in \cT_\rho(\p_+ SM).
$$
Since $\sup_{(x, v), (y, w) \in E}\dist_{\p_+ SM}\big((x, v), (y, w)\big) = \diam_{\p_+ SM}(E) \le \rho$, we can use \eqref{equation: uniform continuity on E} to get $\|f - P_\rho^{\p_+ SM} f\|_{L^\infty(E)} < \varepsilon$ for every $E\in \cT_\rho(\p_+ SM)$. The proof is complete.
\end{proof}

\begin{Proposition}\label{proposition: P_rho converges to Id strongly}
If $f\in L^2_\mu(\p_+ SM)$, for every $\varepsilon > 0$, there exists $\rho_0 > 0$ such that the induced projection $P_\rho^{\p_+ SM}: L^2_\mu(\p_+ SM) \to V_\rho(\p_+ SM)$ of any mesh $\cT_\rho(\p_+ SM)$ on $\p_+ SM$ with $\rho \le \rho_0$ satisfies
$$
    \|f - P_\rho^{\p_+ SM} f\|_{L^2_\mu(\p_+ SM)} < \varepsilon.
$$
\end{Proposition}
\begin{proof}
By the density of $C(\p_+ SM)$ in $L^2_\mu(\p_+ SM)$, there is $f_\varepsilon \in C(\p_+ SM)$ such that
$$
    \|f - f_\varepsilon\|_{L^2_\mu(\p_+ SM)} < \frac{\varepsilon}{2}.
$$
By Lemma~\ref{lemma: approximations of C functions with P_rho}, there exists $\rho_0 > 0$ such that the induced projection $P_\rho^{\p_+ SM}: L^2_\mu(\p_+ SM) \to V_\rho(\p_+ SM)$ of any mesh $\cT_\rho(\p_+ SM)$ on $\p_+ SM$ with $\rho \le \rho_0$ satisfies
$$
    \|f_\varepsilon - P_\rho^{\p_+ SM} f_\varepsilon\|_{L^\infty(\p_+ SM)} < \frac{\varepsilon}{2 \mu(\p_+ SM)^{1/2}}.
$$
Since $P_\rho^{\p_+ SM} f_\varepsilon \in V_\rho(\p_+ SM)$, using the triangle inequality,
$$
    \begin{aligned}
        \inf_{v\in V_\rho(\p_+ SM)}\|f - v\|_{L^2_\mu(\p_+ SM)} &\le \|f - P_\rho^{\p_+ SM} f_\varepsilon\|_{L^2_\mu(\p_+ SM)} \\
        &\le \|f - f_\varepsilon\|_{L^2_\mu(\p_+ SM)} + \|f_\varepsilon - P_\rho^{\p_+ SM} f_\varepsilon\|_{L^2_\mu(\p_+ SM)} \\
        &\le \|f - f_\varepsilon\|_{L^2_\mu(\p_+ SM)} + \|f_\varepsilon - P_\rho^{\p_+ SM} f_\varepsilon\|_{L^\infty(\p_+ SM)} \mu(\p_+ SM)^{1/2} < \varepsilon.
    \end{aligned}
$$
Now, applying Proposition~\ref{proposition: projection properties}(c), we get
$$
    \|f - P_\rho^{\p_+ SM} f\|_{L^2_\mu(\p_+ SM)} = \inf_{v \in V_\rho(\p_+ SM)} \|f - v\|_{L^2_\mu(\p_+ SM)} < \varepsilon
$$
as desired.
\end{proof}

\section{Stability Estimates}\label{section: stability estimates}
Let $(M, g)$ be a compact non-trapping Riemannian manifold with convex boundary. Throughout the paper, we make the following additional assumption.
\begin{Assumption}\label{assumption: simplicity and strict convex admittance of M}
    $(M, g)$ has no conjugate points or admits a smooth strictly convex function when $\dim(M) \ge 3$.
\end{Assumption}

Fix an arbitrary finite-dimensional subspace $\cW\subset L^2(M)$, and let $P_{\cW}:L^2(M)\to\cW$ be the orthogonal projection. Let $\cT_\rho(\p_+ SM)$ be a mesh on $\p_+ SM$, with the induced piecewise constant FES $V_\rho(\p_+ SM)\subset L^2_\mu(\p_+ SM)$ and the associated projection $P_\rho^{\p_+ SM}:L^2_\mu(\p_+ SM)\to V_\rho(\p_+ SM)$. The measurement mesh will need to be sufficiently fine depending on $\cW$.

Finite detector apertures are commonly represented in tomographic reconstruction models by local weighted averages of ray-transform data; see \cite[Sections~V.4.2, V.5.1, V.5.2]{Na1986BGT}, \cite[Section~5.1.2]{KaSl2001SIAM}, \cite[Section~3.2]{RiFa2003}, and \cite[Section~2.1]{YuWa2012}. A further motivation comes from the sampling analysis of \cite[Section~3.2.4]{St2023IP}, which discusses local averaging, and from the references \cite{Co1978PMB, ShLo1974IEEETNC} therein. Motivated by these models, we adopt an idealized measurement model consisting of $\mu$-averages over cells of the full boundary data space $\p_+ SM$.

For $f\in L^2(M)$, the resulting finite collection of measurements consists of the local averages
$$
    \bigg\{\frac{1}{\mu(E)}\int_{E}I_g f \,d\mu\,:\, E\in \cT_\rho(\p_+ SM)\bigg\}.
$$
We represent these averages by the piecewise constant function
$$
    P_\rho^{\p_+ SM}I_g f (x, v) = \sum_{E\in \cT_\rho(\p_+ SM)} \bigg(\frac{1}{\mu(E)}\int_{E}I_g f(y, w)\,d\mu(y, w)\bigg)\mathbf{1}_{E}(x,v),\quad (x,v) \in \p_+ SM.
$$
Each cell average is a bounded linear functional on $L^2_\mu(\p_+ SM)$, so this measurement model is well defined for $f\in L^2(M)$ without requiring pointwise evaluation of $I_g f$.

The following result shows that sufficiently fine measurement partitions preserve stability on $\cW$ and yield a stability estimate for general $f\in L^2(M)$ with an additional term accounting for its approximation error in $\cW$.

\begin{Theorem}\label{theorem: main stability}
Let $(M, g)$ be a compact non-trapping Riemannian manifold with convex boundary and satisfying Assumption~\ref{assumption: simplicity and strict convex admittance of M}. For the fixed finite-dimensional subspace $\cW\subset L^2(M)$, there exists $\rho_0=\rho_0(g,\cW)>0$ such that, for any mesh $\cT_\rho(\p_+ SM)$ with $\rho \le \rho_0$,

$$
    \|f\|_{L^2(M)} \le \frac{2}{\sqrt{3}} C_{g, \cW}\|P_\rho^{\p_+ SM} I_g f\|_{L^2_\mu(\p_+ SM)} + \bigg(1 + \frac{2}{\sqrt{3}} C_{g, \cW}\sqrt{\diam(M)\, \omega_{n - 1}}\,\bigg)\|f - P_{\cW} f \|_{L^2(M)}
$$
for all $f\in L^2(M)$.
\end{Theorem}

\begin{Remark}
    Here $C_{g, \cW}$ is the stability constant that appears in Proposition~\ref{proposition: stability of I_g over W} below. Hence, the price we pay for a finite number of measurements is the increase of this constant by a factor of $2/\sqrt{3}$.
\end{Remark}

\begin{Corollary}\label{corollary: main stability}
Let $(M, g)$ be a compact non-trapping Riemannian manifold with convex boundary and satisfying Assumption~\ref{assumption: simplicity and strict convex admittance of M}. For the fixed finite-dimensional subspace $\cW\subset L^2(M)$, there exists $\rho_0=\rho_0(g,\cW)>0$ such that, for any mesh $\cT_\rho(\p_+ SM)$ with $\rho \le \rho_0$,
$$
    \|f\|_{L^2(M)} \le \frac{2}{\sqrt{3}} C_{g, \cW}\|P_\rho^{\p_+ SM} I_g\imath_{\cW}(f)\|_{L^2_\mu(\p_+ SM)}
$$
for all $f\in \cW$.
\end{Corollary}

For the proof of Theorem~\ref{theorem: main stability}, we need the following two simple but key results, which will be used a few times throughout the paper. To that end, define
$$
    S(\cW) := \{ f\in \cW: \|f\|_{L^2(M)} = 1\},
$$
i.e., the unit sphere in $\cW$. Since $\cW$ is finite dimensional, $S(\cW)$ is compact.

\begin{Proposition}\label{proposition: stability of I_g over W}
    Let $(M, g)$ be a compact non-trapping Riemannian manifold with convex boundary and satisfying Assumption~\ref{assumption: simplicity and strict convex admittance of M}. Then there exists a constant $C_{g, \cW} > 0$ such that 
    $$
        \|f\|_{L^2(M)} \le C_{g, \cW} \|I_g \imath_{\cW}(f)\|_{L^2_\mu(\p_+ SM)}
    $$
    for all $f\in \cW$.
\end{Proposition}

\begin{proof}
    Boundedness of $I_g: L^2(M) \to L^2_\mu(\p_+ SM)$ implies that $\cW\ni f\mapsto \|I_g \imath_{\cW}(f)\|_{L^2_\mu(\p_+ SM)}\in [0, \infty)$ is continuous. Then the compactness of $S(\cW)$ implies the existence of a constant $c_{g, \cW} \ge 0$ such that
    $$
        \|I_g \imath_{\cW}(f)\|_{L^2_\mu(\p_+ SM)} \ge c_{g, \cW} \quad \text{for all} \quad f\in S(\cW).
    $$
    It is clear that $c_{g, \cW} = \inf_{f\in S(\cW)}\|I_g \imath_{\cW}(f)\|_{L^2_\mu(\p_+ SM)}$. The injectivity of $I_g$ guarantees that $c_{g, \cW} > 0$. Setting $C_{g, \cW} = 1 / c_{g, \cW}$, we complete the proof.
\end{proof}

\begin{Remark}
    The proof of Proposition~\ref{proposition: stability of I_g over W} uses Assumption~\ref{assumption: simplicity and strict convex admittance of M} only to ensure the injectivity of $I_g$. Once boundedness of $I_g$ is established, injectivity alone yields the stability estimate on every fixed finite-dimensional subspace $\cW\subset L^2(M)$, by compactness of its unit sphere. In fact, for a prescribed $\cW$, it suffices that $I_g\imath_{\cW}$ be injective, or equivalently, that $\ker I_g\cap\cW=\{0\}$.
\end{Remark}

\begin{Proposition}\label{proposition: operator norm (Id - P_rho)I_g is small if rho is small enough}
    For every $\varepsilon>0$, there exists $\rho_0 > 0$ such that
    $$
        \|(\Id - P_\rho^{\p_+ SM}) I_g\imath_{\cW}\|_{\cW \to L^2_\mu(\p_+ SM)} < \varepsilon
    $$
    for any mesh $\cT_\rho(\p_+ SM)$ with $\rho\le \rho_0$.
\end{Proposition}
\begin{proof}
    By the compactness of $S(\cW)$, there are $f_1,\dots, f_m \in S(\cW)$ such that open balls in $\cW$ of radius $\eta$ and centered at $f_i$, $i = 1, \dots, m$ cover $S(\cW)$, where $\eta > 0$ is chosen so that
    $$
        \eta\, \sqrt{\diam(M)\omega_{n - 1}} < \frac{\varepsilon}{2}.
    $$
    For each $i = 1,\dots, m$, we use Proposition~\ref{proposition: P_rho converges to Id strongly} to find $\rho_i > 0$ such that
    $$
        \|(\Id - P_\rho^{\p_+ SM})I_g \imath_{\cW}(f_i)\|_{L^2_\mu(\p_+ SM)} < \frac{\varepsilon}{2}
    $$
    for any mesh $\cT_\rho(\p_+ SM)$ with $\rho \le \rho_i$. We then define $\rho_0 := \min_{i = 1,\dots, m}\rho_i$.\medskip

    Now, choose an arbitrary mesh $\cT_\rho(\p_+ SM)$ with $\rho \le \rho_0$. For $f\in S(\cW)$, take $i\in \{1, \dots, m\}$ such that $\|f - f_i\|_{L^2(M)} < \eta$. Then
    $$
        \begin{aligned}
            \|(\Id - P_\rho^{\p_+ SM})&I_g \imath_{\cW}(f)\|_{L^2_\mu(\p_+ SM)} \\
            &\le \|(\Id - P_\rho^{\p_+ SM})I_g \imath_{\cW}(f - f_i)\|_{L^2_\mu(\p_+ SM)} + \|(\Id - P_\rho^{\p_+ SM})I_g \imath_{\cW}(f_i)\|_{L^2_\mu(\p_+ SM)} \\
            &\le \sqrt{\diam(M)\omega_{n-1}}\,\eta + \varepsilon/2 < \varepsilon,
        \end{aligned}
    $$
    where we have used Proposition~\ref{proposition: projection properties}(b) and \eqref{equation: explicit bound for operator norm for I_g acting on L^2 functions}. Taking the supremum over all $f\in S(\cW)$ completes the proof.
\end{proof}

\begin{proof}[Proof of Theorem~\ref{theorem: main stability}]
    By Proposition~\ref{proposition: operator norm (Id - P_rho)I_g is small if rho is small enough}, there is $\rho_0 > 0$ such that
    \begin{equation}\label{equation: sufficient operator norm (Id - P_rho)I_g bound}
        \|(\Id - P_\rho^{\p_+ SM}) I_g \imath_{\cW}\|_{\cW \to L^2_\mu(\p_+ SM)} \le \frac{1}{2C_{g, \cW}}
    \end{equation}
    for any mesh $\cT_\rho(\p_+ SM)$ with $\rho\le \rho_0$. Suppose $f\in L^2(M)$ and $h\in \cW$. Then by Proposition~\ref{proposition: stability of I_g over W},
    $$
        \begin{aligned}
            \|h\|_{L^2(M)}^2 &\le C_{g, \cW}^2 \|I_g \imath_{\cW}(h)\|_{L^2_\mu(\p_+ SM)}^2 \\
            &= C_{g, \cW}^2 \|P_\rho^{\p_+ SM} I_g \imath_{\cW}(h)\|_{L^2_\mu(\p_+ SM)}^2 + C_{g, \cW}^2 \|(\Id - P_\rho^{\p_+ SM}) I_g \imath_{\cW}(h)\|_{L^2_\mu(\p_+ SM)}^2 \\
            &\le C_{g, \cW}^2 \|P_\rho^{\p_+ SM} I_g \imath_{\cW}(h)\|_{L^2_\mu(\p_+ SM)}^2 + C_{g, \cW}^2 \|(\Id - P_\rho^{\p_+ SM}) I_g \imath_{\cW}\|_{\cW \to L^2_\mu(\p_+ SM)}^2 \|h\|_{L^2(M)}^2.
        \end{aligned}
    $$
    Substituting \eqref{equation: sufficient operator norm (Id - P_rho)I_g bound} into this estimate, we get
    $$
        \begin{aligned}
            \|h\|_{L^2(M)} &\le \frac{2}{\sqrt{3}} C_{g, \cW}\|P_\rho^{\p_+ SM} I_g \imath_{\cW}(h)\|_{L^2_\mu(\p_+ SM)} \\
            &\le \frac{2}{\sqrt{3}} C_{g, \cW}\|P_\rho^{\p_+ SM} I_g f\|_{L^2_\mu(\p_+ SM)} + \frac{2}{\sqrt{3}} C_{g, \cW}\|P_\rho^{\p_+ SM} I_g (f - \imath_{\cW} h)\|_{L^2_\mu(\p_+ SM)} \\
            &\le \frac{2}{\sqrt{3}} C_{g, \cW}\|P_\rho^{\p_+ SM} I_g f\|_{L^2_\mu(\p_+ SM)} + \frac{2}{\sqrt{3}} C_{g, \cW}\sqrt{\diam(M)\,\omega_{n-1}}\,\|f - h\|_{L^2(M)}
        \end{aligned}
    $$
    where we have used Proposition~\ref{proposition: projection properties}(b) and \eqref{equation: explicit bound for operator norm for I_g acting on L^2 functions}. Therefore,
    $$
        \begin{aligned}
            \|f\|_{L^2(M)} &\le \|h\|_{L^2(M)} + \|f - h\|_{L^2(M)} \\
                &\le \frac{2}{\sqrt{3}} C_{g, \cW}\|P_\rho^{\p_+ SM} I_g f\|_{L^2_\mu(\p_+ SM)} + \bigg(1 + \frac{2}{\sqrt{3}} C_{g, \cW}\sqrt{\diam(M)\,\omega_{n-1}}\,\bigg)\|f - h\|_{L^2(M)}
        \end{aligned}.
    $$
    Taking $h = P_{\cW} f$, we complete the proof.
\end{proof}

\section{Reconstruction Algorithms}\label{section: reconstructions}
The stability estimate derived in the previous section will be used to develop iterative algorithms for reconstructing $f^\dagger\in\cW$ from the data $P_\rho^{\p_+ SM}I_g\imath_{\cW}(f^\dagger)$. For the numerical realization of these algorithms, we will specialize $\cW$ to a piecewise constant finite element subspace of $L^2(M)$, defined below. The convergence proofs in Theorems~\ref{theorem: steepest gd} and~\ref{theorem: cgm} use standard arguments from numerical linear algebra. We include them to clarify how these arguments apply in the present inverse problem setting.

\subsection{$\cW$ as a piecewise constant FES on $M$}
Let $\cT_\delta(M)$ be a mesh, i.e. a finite collection of closed subsets of $M$ with non-empty interiors, satisfying
$$
    M=\bigcup_{F\in\cT_\delta(M)}F,\qquad
    \Vol_g(F\cap G)=0\quad\text{for distinct }F,G\in\cT_\delta(M),\qquad
    \max_{F\in\cT_\delta(M)}\diam_M(F)\le\delta.
$$
The induced piecewise constant FES is
$$
    V_\delta(M):=\span\{\phi_F:F\in\cT_\delta(M)\}\subset L^2(M),
    \qquad \phi_F:=\frac{\mathbf 1_F}{\Vol_g(F)^{1/2}}.
$$
The functions $\{\phi_F\}_{F\in\cT_\delta(M)}$ form an orthonormal basis of $V_\delta(M)$, and $V_\delta(M)\subset L^\infty(M)$ is finite-dimensional.

The orthogonal projection $P_\delta^M:L^2(M)\to V_\delta(M)$ is given by
$$
    P_\delta^M f
    =\sum_{F\in\cT_\delta(M)}(f,\phi_F)_{L^2(M)}\phi_F
    =\sum_{F\in\cT_\delta(M)}\frac{1}{\Vol_g(F)}
    \bigg(\int_F f(x)\,d\Vol_g(x)\bigg)\mathbf 1_F.
$$
As in Proposition~\ref{proposition: projection properties}(a), we can show that
$$
    (P_\delta^M)^*=\imath_\delta^M,\qquad
    (\imath_\delta^M)^*=P_\delta^M,
$$
where $\imath_\delta^M: V_\delta(M)\inclusion L^2(M)$ is the inclusion operator.

For the remainder of this paper, we take $\cW=V_\delta(M)$ in the results of Section~\ref{section: stability estimates}. Accordingly, $P_{\cW}=P_\delta^M$ and $\imath_{\cW}=\imath_\delta^M$.

\subsection{Matrix representations}\label{subsection: matrix representations}
Let $\{\psi_E\}_{E\in \cT_\rho(\p_+ SM)}$ denote the orthonormal basis of $V_\rho(\p_+ SM)$ defined in \eqref{equation: orthonormal basis of FES}. In what follows, with some abuse of notation, we will use the cells of $\cT_\delta(M)$ and $\cT_\rho(\p_+ SM)$ for indexing purposes as well. For example, $f\in V_\delta(M)$ can be represented as $f = \sum_{F\in \cT_\delta(M)} f_F \phi_F$. Applying $P_\rho^{\p_+ SM}I_g\imath_\delta^M$ to $f$, we can write
$$
    P_\rho^{\p_+ SM}I_g\imath_\delta^M(f)(x, v) = \sum_{E\in \cT_\rho(\p_+ SM)} \sum_{F\in \cT_\delta(M)} f_F \big(I_g\imath_\delta^M(\phi_F), \psi_E\big)_{L^2_\mu(\p_+ SM)} \psi_E.
$$
Thus, the map $P_\rho^{\p_+ SM}I_g\imath_\delta^M: V_\delta(M) \to V_\rho(\p_+ SM)$ is represented by the matrix $A_{\rho, \delta}\in\R^{N\times K}$ whose entries are
$$
    (A_{\rho, \delta})_{E, F} = \big(I_g\imath_\delta^M(\phi_F), \psi_E \big)_{L^2_\mu(\p_+ SM)},\qquad E\in \cT_\rho(\p_+ SM), \quad F\in \cT_\delta(M),
$$
where $K = |\cT_\delta(M)|$ and $N = |\cT_\rho(\p_+ SM)|$. In what follows, we also use the notation
$$
    x_f:=(f_F)_{F\in\cT_\delta(M)}\in\R^K
$$
for the coordinate vector of $f\in V_\delta(M)$. Since the two bases are orthonormal,
\begin{equation}\label{equation: coordinate isometries}
    \|f\|_{L^2(M)} = \|x_f\|_{\R^K}, \qquad \|P_\rho^{\p_+ SM}I_g\imath_\delta^M(f)\|_{L^2_\mu(\p_+ SM)} = \|A_{\rho, \delta}\, x_f\|_{\R^N}.
\end{equation}
Moreover, the matrix representation of $(P_\rho^{\p_+ SM}I_g\imath_\delta^M)^*$ in the reversed pair of bases is $A_{\rho, \delta}^T$. Indeed,
$$
    \big((P_\rho^{\p_+ SM}I_g\imath_\delta^M)^*\psi_E,\phi_F\big)_{L^2(M)} = \big(\psi_E, I_g\imath_\delta^M(\phi_F)\big)_{L^2_\mu(\p_+ SM)} = (A_{\rho, \delta})_{EF}.
$$
Consequently, the matrix representation of
$(P_\rho^{\p_+ SM}I_g\imath_\delta^M)^*
P_\rho^{\p_+ SM}I_g\imath_\delta^M$ is $A_{\rho, \delta}^T A_{\rho, \delta}$.

\subsection{The steepest gradient descent}\label{subsection: steepest gd}
For $f^\dagger\in V_\delta(M)$, write $y = P_\rho^{\p_+ SM} I_g\imath_\delta^M(f^\dagger)$ and suppose $f_0\in V_\delta(M)$. Our goal is to reconstruct $f^\dagger$. Consider the following iteration
\begin{equation}\label{equation: steepest gd iteration}
    \begin{aligned}
        r_k &:= (P_\rho^{\p_+ SM} I_g \imath_\delta^M)^*\big(P_\rho^{\p_+ SM} I_g \imath_\delta^M(f_k) - y\big),\\
        \eta_k &:= \frac{\|r_k\|_{L^2(M)}^2}{\|P_\rho^{\p_+ SM} I_g \imath_\delta^M (r_k)\|_{L^2_\mu(\p_+ SM)}^2},\\
        f_{k+1} &:= f_k - \eta_k r_k
    \end{aligned}
\end{equation}
for all integers $k\ge 0$ for which $r_k\ne0$. If $r_k=0$, the iteration is terminated. Note that if $f_k\in V_\delta(M)$, then it is guaranteed $r_k, f_{k+1} \in V_\delta(M)$. Indeed, by the mapping properties of $I_g$, $P_\rho^{\p_+ SM}$, and $\imath_\delta^M$, we have
$$
P_\rho^{\p_+ SM}I_g\imath_\delta^M: V_\delta(M) \to V_\rho(\p_+ SM)
$$
and by Proposition~\ref{proposition: projection properties}(a) and the mapping properties of $I_g$, $P_\delta^M$, and $\imath_\rho^{\p_+ SM}$, we have
$$
(P_\rho^{\p_+ SM} I_g\imath_\delta^M)^* = (\imath_\delta^M)^* I_g^* (P_\rho^{\p_+ SM})^* = P_\delta^M I_g^* \imath_\rho^{\p_+ SM}: V_\rho(\p_+ SM) \to V_\delta(M).
$$
The iteration \eqref{equation: steepest gd iteration} is the steepest gradient descent for the following optimization problem
\begin{equation}\label{equation: optimization problem}
    \min_{f\in V_\delta(M)}\frac{1}{2}\|P_\rho^{\p_+ SM} I_g\imath_\delta^M(f) - y\|_{L^2_\mu(\p_+ SM)}^2.
\end{equation}
Our following result allows us to reconstruct $f^\dagger$ from $P_\rho^{\p_+ SM} I_g f^\dagger$.
\begin{Theorem}\label{theorem: steepest gd}
    Let $(M, g)$ be a compact non-trapping Riemannian manifold with convex boundary and satisfying Assumption~\ref{assumption: simplicity and strict convex admittance of M}. There exists $\rho_0 > 0$ such that if $\cT_\rho(\p_+ SM)$ is a mesh with $\rho \le \rho_0$, and if $f_0\in V_\delta(M)$, then the iteration \eqref{equation: steepest gd iteration} converges at the rate given by
    $$
        \|f_k - f^\dagger\|_{L^2(M)}
        \le 2 C_{g, V_\delta(M)}\sqrt{\frac{\diam(M)\,\omega_{n-1}}{3}}\bigg(\frac{\sigma_{\max}^2 - \sigma_{\min}^2}
        {\sigma_{\max}^2 + \sigma_{\min}^2}\bigg)^k
        \|f_0 - f^\dagger\|_{L^2(M)}
    $$
    for all integers $k\ge 0$, where $\sigma_{\min} > 0$ and $\sigma_{\max}$ are the smallest and largest singular values of the matrix representation of $P_\rho^{\p_+ SM}I_g\imath_\delta^M$. If $r_k = 0$ for some $k$, then
    $f_k = f^\dagger$ and the iteration terminates.
\end{Theorem}

\begin{proof}
    Choose $\rho_0>0$ as in Corollary~\ref{corollary: main stability}. For $\rho\le\rho_0$, the operator $P_\rho^{\p_+ SM}I_g\imath_\delta^M$ is injective, so $A_{\rho, \delta}$ is also injective and $\sigma_{\min} > 0$. Moreover, $A_{\rho, \delta}^T A_{\rho, \delta}$ is symmetric positive definite, and by \cite[Theorem~7.3.5]{HoJo1985}, its smallest and largest eigenvalues are $\sigma_{\min}^2$ and $\sigma_{\max}^2$, respectively.

    Let us write $e_k:=f_k-f^\dagger$. Since $y = P_\rho^{\p_+ SM}I_g\imath_\delta^M(f^\dagger)$, we have
    $$
        r_k=(P_\rho^{\p_+ SM}I_g\imath_\delta^M)^* P_\rho^{\p_+ SM}I_g\imath_\delta^M(e_k).
    $$
    Suppose $r_k = 0$. Injectivity of $P_\rho^{\p_+ SM}I_g\imath_\delta^M$ implies the injectivity of $(P_\rho^{\p_+ SM}I_g\imath_\delta^M)^* P_\rho^{\p_+ SM}I_g\imath_\delta^M$. Hence, $f_k = f^\dagger$. We may therefore assume that $r_k \ne 0$. In that case, the estimate in Corollary~\ref{corollary: main stability}
    shows that $\|P_\rho^{\p_+ SM}I_g\imath_\delta^M(r_k)\|_{L^2_\mu(\p_+ SM)} > 0$, so $\eta_k$ is
    well-defined.

    Subtracting $f^\dagger$ from the last equation in
    \eqref{equation: steepest gd iteration} gives
    $e_{k+1} = e_k - \eta_k r_k$. Moreover,
    $$
        \begin{aligned}
            \big(P_\rho^{\p_+ SM}I_g\imath_\delta^M(e_k),\, P_\rho^{\p_+ SM}I_g\imath_\delta^M(r_k)&\big)_{L^2_\mu(\p_+ SM)} \\
            &= \big((P_\rho^{\p_+ SM}I_g\imath_\delta^M)^*P_\rho^{\p_+ SM}I_g\imath_\delta^M(e_k),r_k\big)_{L^2(M)} = \|r_k\|_{L^2(M)}^2.
        \end{aligned}
    $$
    It follows from the definition of $\eta_k$ that
    \begin{equation}\label{equation: steepest gd exact decrease}
        \begin{aligned}
            \|P_\rho^{\p_+ SM} &I_g \imath_\delta^M (e_{k+1})\|_{L^2_\mu(\p_+ SM)}^2 \\
            &= \|P_\rho^{\p_+ SM} I_g \imath_\delta^M (e_k)\|_{L^2_\mu(\p_+ SM)}^2 + \eta_k^2 \|P_\rho^{\p_+ SM} I_g \imath_\delta^M (r_k)\|_{L^2_\mu(\p_+ SM)}^2 - 2\eta_k \|r_k\|_{L^2(M)}^2 \\
            &= \|P_\rho^{\p_+ SM} I_g \imath_\delta^M (e_k)\|_{L^2_\mu(\p_+ SM)}^2 - \frac{\|r_k\|_{L^2(M)}^4}{\|P_\rho^{\p_+ SM} I_g \imath_\delta^M (r_k)\|_{L^2_\mu(\p_+ SM)}^2}.
        \end{aligned}
    \end{equation}
    By the matrix representation properties established above, $x_{r_k} = A_{\rho, \delta}^T A_{\rho, \delta} x_{e_k}$. Then \eqref{equation: coordinate isometries} yields
    $$
        \|r_k\|_{L^2(M)}^2 = (x_{r_k}, x_{r_k})_{\R^K}, \qquad \|P_\rho^{\p_+ SM} I_g \imath_\delta^M(r_k)\|_{L^2_\mu(\p_+ SM)}^2 = (A_{\rho, \delta}^T A_{\rho, \delta} \, x_{r_k}, x_{r_k})_{\R^K},
    $$
    and, since $x_{e_k} = (A_{\rho, \delta}^T A_{\rho, \delta})^{-1} x_{r_k}$,
    $$
        \|P_\rho^{\p_+ SM} I_g \imath_\delta^M(e_k)\|_{L^2_\mu(\p_+ SM)}^2 = (A_{\rho, \delta}^T A_{\rho, \delta} x_{e_k}, x_{e_k})_{\R^K} = \big((A_{\rho, \delta}^T A_{\rho, \delta})^{-1} x_{r_k}, x_{r_k}\big)_{\R^K}.
    $$
    Applying \cite[Theorem~7.4.41]{HoJo1985} gives
    $$
        (A_{\rho, \delta}^T A_{\rho, \delta} \, x_{r_k}, x_{r_k})_{\R^K} \big((A_{\rho, \delta}^T A_{\rho, \delta})^{-1} x_{r_k}, x_{r_k}\big)_{\R^K} \le \frac{(\sigma_{\max}^2 + \sigma_{\min}^2)^2}{4\sigma_{\max}^2 \sigma_{\min}^2}\|x_{r_k}\|_{\R^K}^4.
    $$
    Therefore,
    $$
        \frac{\|r_k\|_{L^2(M)}^4}
        {\|P_\rho^{\p_+ SM} I_g \imath_\delta^M (r_k)\|_{L^2_\mu(\p_+ SM)}^2
         \|P_\rho^{\p_+ SM} I_g \imath_\delta^M (e_k)\|_{L^2_\mu(\p_+ SM)}^2}
        \ge \frac{4\sigma_{\max}^2\sigma_{\min}^2}{(\sigma_{\max}^2+\sigma_{\min}^2)^2}.
    $$
    Combining this estimate with
    \eqref{equation: steepest gd exact decrease}, we obtain
    $$
        \begin{aligned}
            \|P_\rho^{\p_+ SM} I_g \imath_\delta^M (e_{k+1})\|_{L^2_\mu(\p_+ SM)}^2 &\le \bigg(1-\frac{4\sigma_{\max}^2\sigma_{\min}^2}{(\sigma_{\max}^2+\sigma_{\min}^2)^2}\bigg) \|P_\rho^{\p_+ SM} I_g \imath_\delta^M(e_k)\|_{L^2_\mu(\p_+ SM)}^2 \\
            &= \bigg(\frac{\sigma_{\max}^2 - \sigma_{\min}^2}{\sigma_{\max}^2 + \sigma_{\min}^2}\bigg)^2 \|P_\rho^{\p_+ SM} I_g \imath_\delta^M(e_k)\|_{L^2_\mu(\p_+ SM)}^2.
        \end{aligned}
    $$
    Iterating this we get
    $$
        \|P_\rho^{\p_+ SM}I_g\imath_\delta^M(f_k-f^\dagger)\|_{L^2_\mu(\p_+ SM)}
        \le \bigg(\frac{\sigma_{\max}^2 - \sigma_{\min}^2}
        {\sigma_{\max}^2 + \sigma_{\min}^2}\bigg)^k
        \|P_\rho^{\p_+ SM}I_g\imath_\delta^M(f_0-f^\dagger)\|_{L^2_\mu(\p_+ SM)}.
    $$
    Finally, applying the estimate in Corollary~\ref{corollary: main stability} to the left side of the above estimate and \eqref{equation: explicit bound for operator norm for I_g acting on L^2 functions} together with Proposition~\ref{proposition: projection properties}(b) to the right side, we get the statement of the theorem.
\end{proof}

The reconstruction algorithm is summarized in the Algorithm~\ref{algorithm: steepest gradient descent} pseudo-code.

\begin{algorithm}[htbp]
    \caption{Steepest gradient descent}
    \label{algorithm: steepest gradient descent}
    \begin{algorithmic}[1]
        \Require $P_\rho^{\p_+ SM}I_g\imath_\delta^M:V_\delta(M)\to V_\rho(\p_+ SM)$, data $y$, initial iterate $f_0\in V_\delta(M)$, tolerance $\varepsilon \ge 0$
        \State $k\gets 0$
        \While{$\|P_\rho^{\p_+ SM}I_g\imath_\delta^M(f_k) - y\|_{L^2_\mu(\p_+ SM)} > \varepsilon$}
            \State $r_k\gets (P_\rho^{\p_+ SM}I_g\imath_\delta^M)^*\big(P_\rho^{\p_+ SM}I_g\imath_\delta^M(f_k) - y\big)$
            \If{$r_k=0$}
                \State \textbf{break}
            \EndIf
            \State $\displaystyle\eta_k\gets \frac{\|r_k\|_{L^2(M)}^2}{\|P_\rho^{\p_+ SM}I_g\imath_\delta^M(r_k)\|_{L^2_\mu(\p_+ SM)}^2}$
            \State $f_{k+1}\gets f_k - \eta_k r_k$
            \State $k\gets k + 1$
        \EndWhile\\
        $\hat f\gets f_k$
    \end{algorithmic}
\end{algorithm}

\subsection{The conjugate gradient method}\label{subsection: cgm}
For $f^\dagger\in V_\delta(M)$, write $y = P_\rho^{\p_+ SM} I_g\imath_\delta^M(f^\dagger)$ and suppose $f_0\in V_\delta(M)$. Consider the following iteration
\begin{equation}\label{equation: cgm iteration}
    \begin{aligned}
        r_k &:= (P_\rho^{\p_+ SM} I_g \imath_\delta^M)^*\big(P_\rho^{\p_+ SM} I_g \imath_\delta^M(f_k) - y\big),\\
        d_k &:= r_k - \sum_{j = 0}^{k - 1}\frac{\big(P_\rho^{\p_+ SM} I_g \imath_\delta^M(d_j),\, P_\rho^{\p_+ SM} I_g \imath_\delta^M(r_k)\big)_{L^2_\mu(\p_+ SM)}}{\|P_\rho^{\p_+ SM} I_g \imath_\delta^M(d_j)\|_{L^2_\mu(\p_+ SM)}^2}d_j,\\
        \alpha_k &:= \frac{(r_k,d_k)_{L^2(M)}}{\|P_\rho^{\p_+ SM} I_g \imath_\delta^M(d_k)\|_{L^2_\mu(\p_+ SM)}^2},\\
        f_{k+1} &:= f_k - \alpha_k d_k
    \end{aligned}
\end{equation}
for all integers $k\ge 0$ for which $r_k\ne0$. If $r_k=0$, the iteration is terminated. The sum defining $d_0$ is understood to be empty, so $d_0=r_0$. Note that if $f_k\in V_\delta(M)$, then, as in Section~\ref{subsection: steepest gd}, it is guaranteed that $r_k,d_k,f_{k+1}\in V_\delta(M)$. The iteration \eqref{equation: cgm iteration} is the conjugate gradient method applied to the optimization problem \eqref{equation: optimization problem}. Our following result allows us to reconstruct $f^\dagger$ from $P_\rho^{\p_+ SM} I_g f^\dagger$.

\begin{Theorem}\label{theorem: cgm}
    Let $(M,g)$ be a compact non-trapping Riemannian manifold with convex
    boundary and satisfying Assumption~\ref{assumption: simplicity and strict convex admittance of M}.
    There exists $\rho_0>0$ such that if $\cT_\rho(\p_+ SM)$ is a mesh
    with $\rho\le\rho_0$ and $f_0\in V_\delta(M)$, then, in exact
    arithmetic, the iteration \eqref{equation: cgm iteration} reaches
    $f^\dagger$ after at most $K = \dim V_\delta(M)$ iterations. Namely, $d_k=0$ for some $k \le K$, in which case $f_k=f^\dagger$ and the iteration terminates. Moreover, the same convergence bound holds as in Theorem~\ref{theorem: steepest gd}.
\end{Theorem}

\begin{proof}
    Choose $\rho_0>0$ as in Corollary~\ref{corollary: main stability}. For $\rho\le\rho_0$, the operator $P_\rho^{\p_+ SM}I_g\imath_\delta^M$ is injective, so $A_{\rho, \delta}$ is also injective and $\sigma_{\min} > 0$. Moreover, $A_{\rho, \delta}^T A_{\rho, \delta}$ is symmetric positive definite, and by \cite[Theorem~7.3.5]{HoJo1985}, its smallest and largest eigenvalues are $\sigma_{\min}^2$ and $\sigma_{\max}^2$, respectively.

    Let us write $e_k:=f_k-f^\dagger$. Since $y = P_\rho^{\p_+ SM}I_g\imath_\delta^M(f^\dagger)$, we have
    \begin{equation}\label{equation: cgm residual}
        r_k = (P_\rho^{\p_+ SM}I_g\imath_\delta^M)^* P_\rho^{\p_+ SM}I_g\imath_\delta^M(e_k).
    \end{equation}
    Consider the symmetric bilinear form $B: V_\delta(M)\times V_\delta(M) \to \R$ defined as
    $$
        B(f, h) := \big(P_\rho^{\p_+ SM}I_g\imath_\delta^M(f), P_\rho^{\p_+ SM}I_g\imath_\delta^M(h)\big)_{L^2_\mu(\p_+ SM)},\qquad f, h \in V_\delta(M).
    $$
    According to Corollary~\ref{corollary: main stability}, $B$ is strictly positive definite. We say that $f, h \in V_\delta(M)$ are \emph{$B$-orthogonal} if $B(f, h) = 0$. The second formula in \eqref{equation: cgm iteration} says precisely that each $d_k$ is obtained by $B$-orthogonalizing it against $d_0, \dots, d_{k-1}$. Thus, any nonzero $d_i$ and $d_j$ with $i\neq j$ are mutually $B$-orthogonal.

    Next, it is easy to check that if $d_0,\dots, d_k \in V_\delta(M)$ are nonzero and $B$-orthogonal, then they are linearly independent. Therefore, $k$ can be at most $K - 1$ so that $d_j \neq 0$ for all $j \le k$.

    Now we show by induction on $k \ge 0$ that
    \begin{equation}\label{equation: cgm gradient direction orthogonality}
        (r_k,d_j)_{L^2(M)}=0,\qquad 0\le j<k.
    \end{equation}
    There is nothing to check in the case $k=0$. Suppose the statement is true for $k - 1$. Since $e_k = e_{k-1} - \alpha_{k-1} d_{k-1}$, we have
    $r_k = r_{k-1} - \alpha_{k-1} (P_\rho^{\p_+ SM}I_g\imath_\delta^M)^* P_\rho^{\p_+ SM}I_g\imath_\delta^M(d_{k-1})$. The induction hypothesis for $j < k - 1$ and the mutual $B$-orthogonality of $\{d_0, \dots, d_{k - 1}\}$ give
    $$
        (r_k, d_j)_{L^2(M)} = (r_{k-1}, d_j)_{L^2(M)}-\alpha_{k-1} B(d_{k-1}, d_j) = 0,\qquad j < k - 1.
    $$
    For $j = k - 1$, the definition of $\alpha_{k-1}$ gives
    $$
        (r_k, d_{k-1})_{L^2(M)} = (r_{k-1}, d_{k-1})_{L^2(M)} - \alpha_{k-1} \|P_\rho^{\p_+ SM}I_g\imath_\delta^M(d_{k-1})\|_{L^2_\mu(\p_+ SM)}^2 = 0.
    $$
    This proves \eqref{equation: cgm gradient direction orthogonality}.

    Finally, suppose that $d_k = 0$. Then the second formula in \eqref{equation: cgm iteration} implies that $r_k$ belongs to $\span\{d_0, \dots, d_{k - 1}\}$, while \eqref{equation: cgm gradient direction orthogonality} says that it is also orthogonal to this space. Therefore, $r_k = 0$. Since the operator $(P_\rho^{\p_+ SM}I_g\imath_\delta^M)^* P_\rho^{\p_+ SM}I_g\imath_\delta^M$ is injective, \eqref{equation: cgm residual} implies $e_k=0$, and hence
    $f_k=f^\dagger$.
    
    It remains to estimate the convergence rate. By \eqref{equation: cgm gradient direction orthogonality} and the second formula in \eqref{equation: cgm iteration}, we have
    \begin{equation}\label{equation: r_k inner d_k is d_k norm squared}
        (r_k, d_k)_{L^2(M)} = \|r_k\|^2_{L^2(M)}.
    \end{equation}
    Moreover, the mutual $B$-orthogonality of $\{d_0, \dots, d_k\}$ and the second formula in \eqref{equation: cgm iteration} imply that $d_k$ is $B$-orthogonal to $r_k - d_k$. Therefore,
    \begin{equation}\label{equation: r_k norm is greater than d_k norm}
        \begin{aligned}
            \|P_\rho^{\p_+ SM}I_g\imath_\delta^M(r_k)\|^2_{L_\mu^2(\p_+ SM)} &= \|P_\rho^{\p_+ SM}I_g\imath_\delta^M(d_k)\|^2_{L_\mu^2(\p_+ SM)} + \|P_\rho^{\p_+ SM}I_g\imath_\delta^M(r_k - d_k)\|^2_{L_\mu^2(\p_+ SM)} \\
            &\ge \|P_\rho^{\p_+ SM}I_g\imath_\delta^M(d_k)\|^2_{L_\mu^2(\p_+ SM)}.
        \end{aligned}
    \end{equation}
    The rest of the proof is exactly the same as in the proof of Theorem~\ref{theorem: steepest gd} with $\eta_k$ replaced by $\alpha_k$ and using \eqref{equation: r_k inner d_k is d_k norm squared} and \eqref{equation: r_k norm is greater than d_k norm}. The proof is thus complete.
\end{proof}

The Algorithm~\ref{algorithm: conjugate gradient method} pseudo-code provides an explicit description of the reconstruction algorithm.

\begin{algorithm}
    \caption{Conjugate gradient method}
    \label{algorithm: conjugate gradient method}
    \begin{algorithmic}[1]
        \Require $P_\rho^{\p_+ SM}I_g\imath_\delta^M:V_\delta(M)\to V_\rho(\p_+ SM)$, data $y$, initial iterate $f_0\in V_\delta(M)$, tolerance $\varepsilon \ge 0$
        \State $k\gets0$
        \While{$\|P_\rho^{\p_+ SM}I_g\imath_\delta^M(f_k) - y\|_{L^2_\mu(\p_+ SM)}>\varepsilon$}
            \State $r_k\gets (P_\rho^{\p_+ SM}I_g\imath_\delta^M)^*\big(P_\rho^{\p_+ SM}I_g\imath_\delta^M(f_k) - y\big)$
            \If{$r_k = 0$}
                \State \textbf{break}
            \EndIf
            \State $d_k\gets r_k$
            \For{$j = 0,\ldots, k - 1$}
                \State $\displaystyle
                    c_{k,j}\gets \frac{\big(P_\rho^{\p_+ SM}I_g\imath_\delta^M(d_j), P_\rho^{\p_+ SM}I_g\imath_\delta^M(r_k)\big)_{L^2_\mu(\p_+ SM)}}{\|P_\rho^{\p_+ SM}I_g\imath_\delta^M(d_j)\|_{L^2_\mu(\p_+ SM)}^2}$
                \State $d_k\gets d_k - c_{k,j} d_j$
            \EndFor
            \If{$d_k = 0$}
                \State \textbf{break}
            \EndIf
            \State $\displaystyle\alpha_k\gets \frac{(r_k,d_k)_{L^2(M)}}
                {\|P_\rho^{\p_+ SM}I_g\imath_\delta^M(d_k)\|_{L^2_\mu(\p_+ SM)}^2}$
            \State $f_{k+1}\gets f_k - \alpha_k d_k$
            \State $k\gets k + 1$
        \EndWhile\\
        $\hat f\gets f_k$
    \end{algorithmic}
\end{algorithm}

\section{Numerical Experiments}\label{section: experiments}
All experiments and simulations in the present paper are implemented by the author in Python\footnote{The code is available at: \href{https://github.com/ya2405/inverse-problems}{\nolinkurl{https://github.com/ya2405/inverse-problems}}}, with the AI tool used as a coding assistant, as disclosed above. The author reviewed and validated all generated or suggested code and takes responsibility for the implementation and results.

\subsection{Experimental setup}\label{subsection: experimental setup}
We present reconstruction experiments on the Euclidean unit ball $M=\{x=(x_1,\dots,x_n)\in\R^n:|x|^2\le1\}$, with $n=2$ or $n=3$, endowed with a conformal metric $g(x)=e^{2\lambda(x)}|dx|^2$. In both dimensions, we consider the spherical stereographic, Poincar\'e, and focusing lens geometries, with logarithmic conformal factors
\begin{equation}\label{equation: lambda functions}
    \begin{aligned}
        \lambda_1(x) &= \log\bigg(\frac{2R_1^2}{|x|^2+R_1^2}\bigg),\qquad R_1=2.0,\\
        \lambda_2(x) &= \log\bigg(\frac{2R_2^2}{R_2^2-|x|^2}\bigg),\qquad R_2=1.1,\\
        \lambda_3(x) &= 0.6\exp\bigg(-\frac{(x_1-0.2)^2+\sum_{j=2}^n x_j^2}{2\sigma^2}\bigg),\qquad \sigma=0.25.
    \end{aligned}
\end{equation}
The volume form in these coordinates is $d\Vol_g=e^{n\lambda}dx_1\cdots dx_n$. The first two metrics have constant sectional curvatures $1/R_1^2$ and $-1/R_2^2$, respectively, and have no conjugate points. The third is a lens focusing away from the origin and has conjugate points. All three geometries are non-trapping on $M$ and have strictly convex boundaries.

\subsubsection*{Discretizations}
In each dimension, the mesh $\cT_\delta(M)$ is obtained from an equispaced Cartesian grid of $[-1,1]^n$ by retaining cells whose centers lie in $M$. The measurement mesh $\cT_\rho(\p_+ SM)$ is a uniform product grid in the angular coordinates specified below. Both the unknowns and the data are represented in the orthonormal bases
$$
    \phi_F=\frac{\mathbf 1_F}{\Vol_g(F)^{1/2}},\qquad F\in\cT_\delta(M),
    \qquad
    \psi_E=\frac{\mathbf 1_E}{\mu(E)^{1/2}},\qquad E\in\cT_\rho(\p_+ SM),
$$
as in Section~\ref{subsection: matrix representations}. Cell measures are evaluated by midpoint quadrature. In particular, with $x^F$ denoting the center of $F$,
$$
    \Vol_g(F)\approx e^{n\lambda(x^F)}\prod_{j=1}^n\Delta x_j,\qquad (x_f)_F=(f,\phi_F)_{L^2(M)}\approx f(x^F)\Vol_g(F)^{1/2}.
$$
The formulas for $\mu(E)$ and the grid sizes depend on the dimension and are given in the corresponding subsections.

\subsubsection*{Calculation of $P_\rho^{\p_+ SM}I_g \imath_\delta^M$}
The matrix representation $A_{\rho,\delta}$ of $P_\rho^{\p_+ SM}I_g\imath_\delta^M$ has entries
$$
    (A_{\rho,\delta})_{E,F} = \frac{1}{\mu(E)^{1/2}\Vol_g(F)^{1/2}} \int_E\int_0^{\tau(x,v)}\mathbf 1_F(\gamma_{x,v}(t))\,dt\,d\mu(x,v), \qquad E\in\cT_\rho(\p_+ SM), \quad F\in\cT_\delta(M).
$$
For each measurement cell $E\in \cT_\rho(\p_+ SM)$, let $(x^E, v^E)\in E$ be the initial state corresponding to its center in the chosen angular coordinates. The second-order Heun method, with a time step $\Delta t=0.01$, produces successive numerical points $x_{E,0},x_{E,1},\dots,x_{E,N_E}$ along $\gamma_{x^E, v^E}$. The length of the portion of this geodesic lying in $F$ is approximated by
$$
    \int_0^{\tau(x^E, v^E)}\mathbf 1_F(\gamma_{x^E, v^E}(t))\,dt
    \approx\Delta t\sum_{m=0}^{N_E}\mathbf 1_F(x_{E,m})
    =\Delta t\,n_{E,F},
$$
where $n_{E,F}$ is the number of sampled points lying in $F$. Midpoint quadrature over $E$ then gives the sparse matrix entries
$$
    (A_{\rho,\delta})_{E,F}
    \approx\frac{\mu(E)^{1/2}}{\Vol_g(F)^{1/2}}\,\Delta t\,n_{E,F}.
$$
Therefore,
$$
    (A_{\rho,\delta}x_f)_E \approx \sum_{F\in\cT_\delta(M)} \frac{\mu(E)^{1/2}}{\Vol_g(F)^{1/2}}\,\Delta t\,n_{E,F}(x_f)_F \approx \mu(E)^{1/2}\sum_{F\in\cT_\delta(M)}\Delta t\,n_{E,F}f(x^F).
$$

\subsubsection*{Reconstruction and visualization}
The ground-truth $f^\dagger$ is the projection of the corresponding \emph{modified Shepp-Logan/Toft} phantom onto $V_\delta(M)$, with coefficients evaluated using midpoint quadrature. We form the discrete data $y=A_{\rho,\delta}x_{f^\dagger}$ using the matrix approximation of $P_\rho^{\p_+ SM}I_g\imath_\delta^M$ described above. For visualization, the phantom and reconstruction coefficients are divided by $\Vol_g(F)^{1/2}$ to recover function values. Similarly, the displayed ray data use $\sum_{F\in\cT_\delta(M)}\Delta t\,n_{E,F}f^\dagger(x^F)$ rather than the coefficients in the measurement basis.

Starting from $f_0=0$, we apply Algorithm~\ref{algorithm: steepest gradient descent} (SGD) and Algorithm~\ref{algorithm: conjugate gradient method} (CGM). Both methods use the stopping rule
\begin{equation}\label{equation: stopping rule}
    \|P_\rho^{\p_+ SM}I_g\imath_\delta^M(f_k)-y\|_{L^2_\mu(\p_+ SM)}\le \varepsilon.
\end{equation}

\subsection{2D setting}
In 2D, $M$ is the unit disk, and the metrics are given by \eqref{equation: lambda functions} with $n=2$. Figure~\ref{figure: geodesics and sinograms} illustrates sample geodesics for each geometry.

We use \emph{fan-beam coordinates} on $\p_+ SM$, defined as
$$
    (\theta, \alpha) \mapsto (x(\theta), v(\theta, \alpha)),\qquad 0\le \theta< 2\pi,\quad -\frac{\pi}{2}< \alpha< \frac{\pi}{2},
$$
where
$$
    x(\theta) := (\cos(\theta), \sin(\theta)),\qquad v(\theta, \alpha) := e^{-\lambda(x(\theta))}(\cos(\theta + \pi + \alpha), \sin(\theta + \pi + \alpha)).
$$
The measure $d\mu$ in these coordinates will take the form $d\mu(x, v) = e^{\lambda(x(\theta))}\cos(\alpha)\,d\alpha\,d\theta$.

\subsubsection*{Discretizations}
The mesh $\cT_\delta(M)$ is obtained from a $128\times128$ equispaced Cartesian grid of $[-1, 1]\times [-1, 1]$ using the cell-selection rule in Section~\ref{subsection: experimental setup}. This gives $K=|\cT_\delta(M)|=12892$ with $\Delta x_1=\Delta x_2=1/64$.

The mesh $\cT_\rho(\p_+ SM)$ consists of a $256\times 192$ uniform Cartesian grid of $[0, 2\pi)\times (-\pi/2, \pi/2)$.  Hence,
$N=|\cT_\rho(\p_+ SM)|=49152$. The midpoint weights are
$$
    \mu(E)\approx \cos(\alpha^E)\,e^{\lambda(x(\theta^E))}\,\Delta\theta\,\Delta\alpha,\qquad \Delta\theta=\frac{\pi}{128},\quad \Delta\alpha=\frac{\pi}{192},
$$
where $(\theta^E, \alpha^E)$ is the center of $E\in \cT_\rho(\p_+ SM)$. The corresponding initial state for the center ray is $(x^E, v^E)=(x(\theta^E),v(\theta^E, \alpha^E))$.

\subsection{3D setting}
We next consider the unit ball in $\R^3$, with the metrics in \eqref{equation: lambda functions} specialized to $n=3$. In particular, the focusing lens is centered at $(0.2, 0, 0)$.

We use the 3D version of fan-beam coordinates on $\p_+ SM$, which we now describe. First, we use the spherical coordinates for the boundary $\p M$ defined as
$$
    (\theta,\varphi)\mapsto x(\theta,\varphi),\qquad 0< \theta <\pi,\quad 0\le \varphi< 2\pi,
$$
where
$$
    x(\theta,\varphi) := (\sin(\theta)\cos(\varphi),\sin(\theta)\sin(\varphi),\cos(\theta)).
$$
Next, consider the Euclidean orthonormal spherical frame $(e_r, e_\theta, e_\varphi)$ at the boundary point $x(\theta,\varphi)$
$$
    \begin{aligned}
        e_r &:= x(\theta,\varphi) = (\sin(\theta)\cos(\varphi), \sin(\theta)\sin(\varphi), \cos(\theta)), \\
        e_\theta &:= \frac{\p x}{\p \theta}(\theta,\varphi) = (\cos(\theta)\cos(\varphi), \cos(\theta)\sin\varphi, -\sin(\theta)), \\
        e_\varphi &:=\frac{1}{\sin(\theta)} \frac{\p x}{\p \varphi}(\theta,\varphi) = (-\sin(\varphi), \cos(\varphi), 0).
    \end{aligned}
$$
Here $e_r$ is the outward Euclidean unit normal to $\p M$, while $e_\theta$ and $e_\varphi$ are Euclidean unit tangent vectors in the directions of increasing $\theta$ and $\varphi$, respectively. Then the 3D fan-beam coordinates on $\p_+ SM$ will be defined as
$$
    (\theta,\varphi,\alpha,\gamma)\mapsto (x(\theta,\varphi),v(\theta,\varphi,\alpha,\gamma)),\qquad 0< \theta< \pi,\quad 0\le \varphi< 2\pi,\quad 0< \alpha< \frac{\pi}{2},\quad 0\le \gamma< 2\pi
$$
where
$$
    v(\theta,\varphi,\alpha,\gamma) := e^{-\lambda(x(\theta, \varphi))}\big(-\cos(\alpha) e_r + \sin(\alpha) \cos(\gamma) e_\theta + \sin(\alpha) \sin(\gamma) e_\varphi \big).
$$
Note that $\alpha$ is the angle with the inward normal.

Since the boundary area element is $e^{2\lambda(x(\theta, \varphi))}\sin(\theta)\,d\theta\,d\varphi$ and the unit-sphere angular element is $\sin\alpha\,d\alpha\,d\gamma$, the influx measure becomes
\begin{equation}\label{equation: influx measure spatial}
    d\mu = e^{2\lambda(x(\theta,\varphi))}\sin(\theta)\cos(\alpha)\sin(\alpha)\,d\theta\,d\varphi\,d\alpha\,d\gamma.
\end{equation}

\subsubsection*{Discretizations}
The mesh $\cT_\delta(M)$ is obtained from a $64\times64\times64$ equispaced Cartesian grid of $[-1, 1]\times [-1, 1]\times [-1, 1]$ using the same cell-selection rule. This gives $K=|\cT_\delta(M)|=137376$, with $\Delta x_1=\Delta x_2=\Delta x_3=1/32$.

For the mesh $\cT_\rho(\p_+ SM)$, we use a $32\times 64\times 32\times 64$ product grid in $(\theta,\varphi,\alpha,\gamma)$. Thus, $N = |\cT_\rho(\p_+ SM)| = 4194304$, and
$$
    \Delta\theta=\frac{\pi}{32},\quad \Delta\varphi=\frac{\pi}{32},\quad \Delta\alpha=\frac{\pi}{64},\quad \Delta\gamma=\frac{\pi}{32}.
$$
For a cell $E\in \cT_\rho(\p_+ SM)$ centered at $(\theta^E, \varphi^E, \alpha^E, \gamma^E)$, let $x_E=x(\theta^E,\varphi^E)$. The midpoint weights are
$$
    \mu(E)\approx e^{2\lambda(x_E)}\sin(\theta^E)\cos(\alpha^E)\sin(\alpha^E)\,\Delta\theta\,\Delta\varphi\,\Delta\alpha\,\Delta\gamma.
$$
The corresponding initial state for the center ray is $(x^E, v^E) = (x(\theta^E,\varphi^E), v(\theta^E,\varphi^E,\alpha^E,\gamma^E))$.

\subsection{Experiments}
Figure~\ref{figure: shepp-logan phantom planar} shows the piecewise constant 2D phantom, and Figure~\ref{figure: geodesics and sinograms} presents its sinograms for the three metrics defined in \eqref{equation: lambda functions}. Figure~\ref{figure: shepp-logan phantom spatial} illustrates the slices of the piecewise constant 3D phantom, while the slices of its sinograms are shown in Figure~\ref{figure: slices of sinograms spatial}. We use the reconstruction and visualization conventions of Section~\ref{subsection: experimental setup}.

\begin{figure}
    \centering
    \includegraphics[width=5cm]{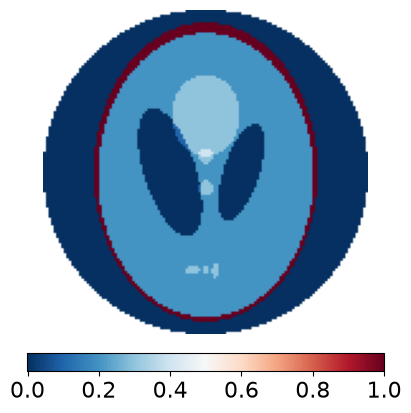}
    \caption{Piecewise constant projection of 2D Shepp-Logan/Toft phantom onto $128\times 128$ uniform Cartesian mesh $\cT_\delta(M)$.}
    \label{figure: shepp-logan phantom planar}
\end{figure}

\begin{figure}
    \centering
    \includegraphics[width=15cm]{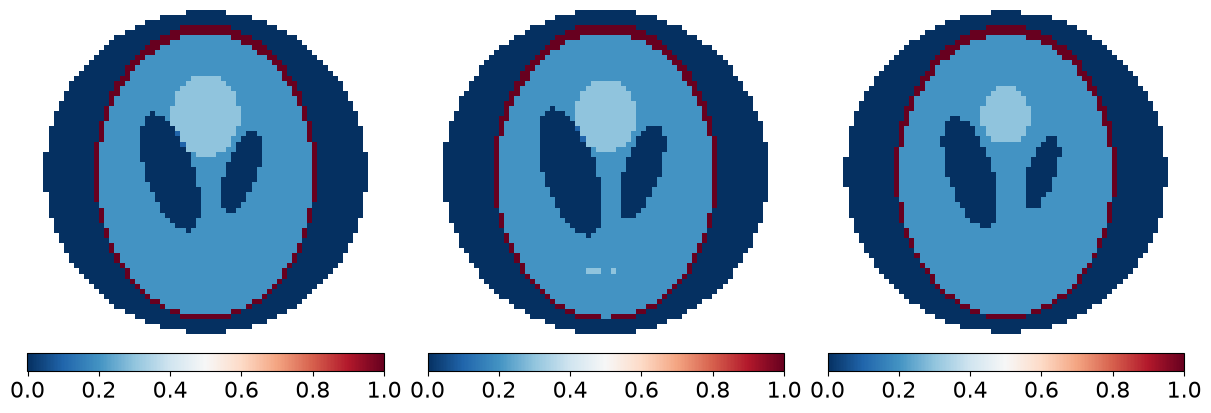}
    \caption{Slices of piecewise constant projection of 3D Shepp-Logan/Toft phantom onto $64\times 64\times 64$ uniform Cartesian mesh $\cT_\delta(M)$. From left to right: slices at $z \approx -0.125$, $z \approx 0.0$ and $z \approx 0.125$.}
    \label{figure: shepp-logan phantom spatial}
\end{figure}

\begin{figure}
    \centering
    \includegraphics[width=15cm]{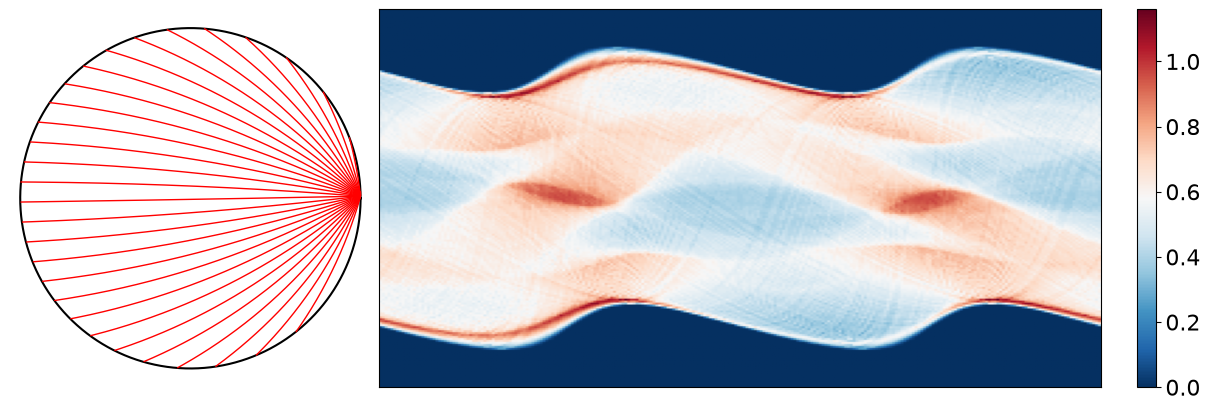}\\
    \includegraphics[width=15cm]{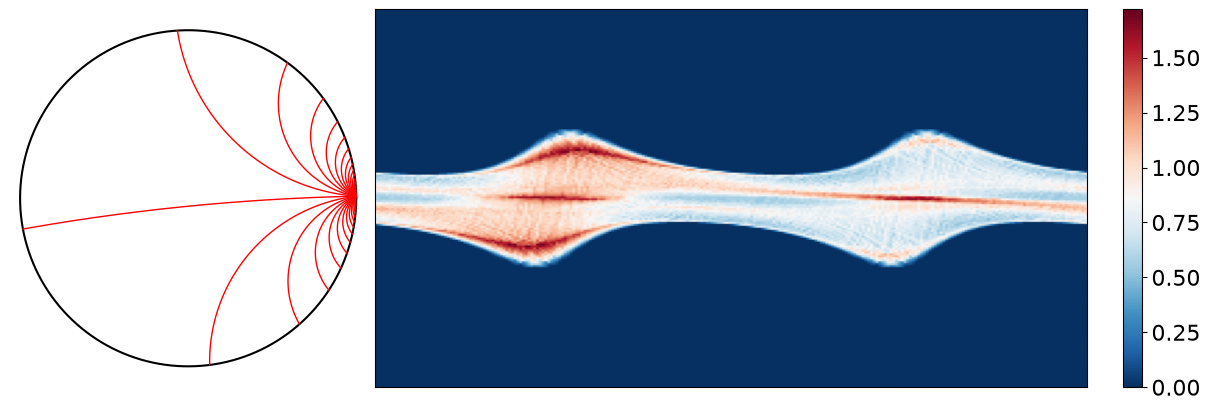}\\
    \includegraphics[width=15cm]{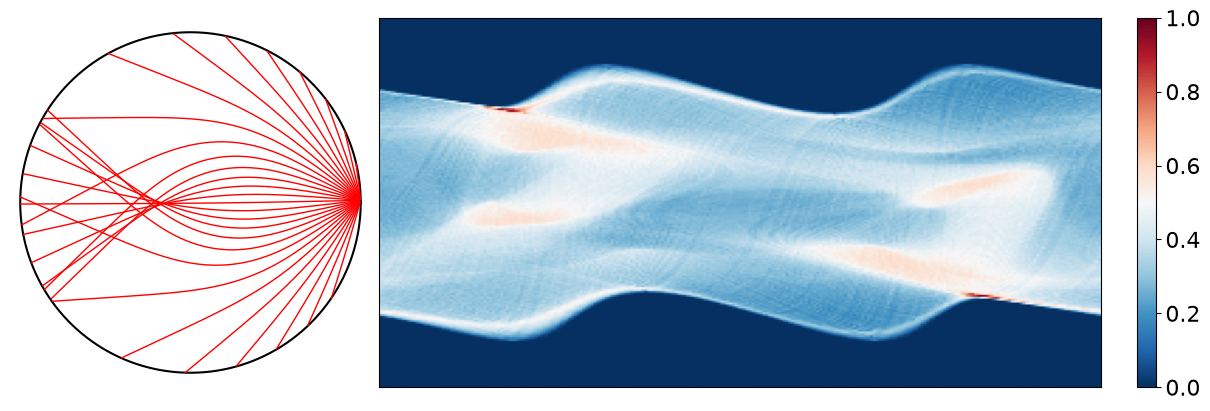}
    \caption{Some geodesics inside $M$ and sinogram corresponding to $f^\dagger$. Rows 1, 2 and 3 correspond to 2D versions of Experiments~\ref{experiment: stereographic},~\ref{experiment: poincare} and \ref{experiment: focusing lens}, respectively.}
    \label{figure: geodesics and sinograms}
\end{figure}

\begin{figure}
    \centering
    \includegraphics[width=15cm]{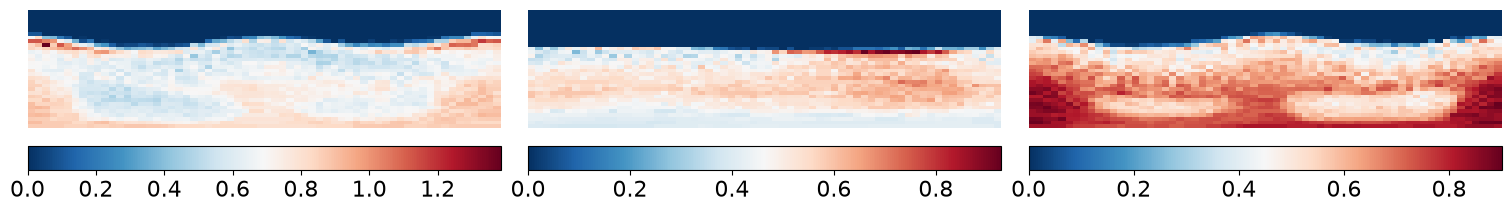}\\
    \includegraphics[width=15cm]{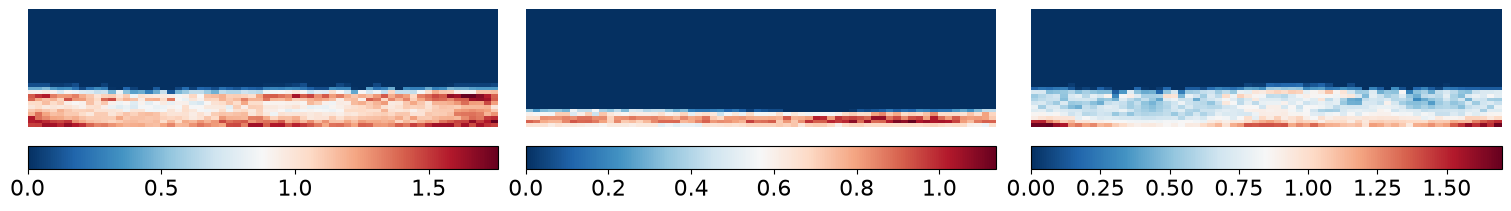}\\
    \includegraphics[width=15cm]{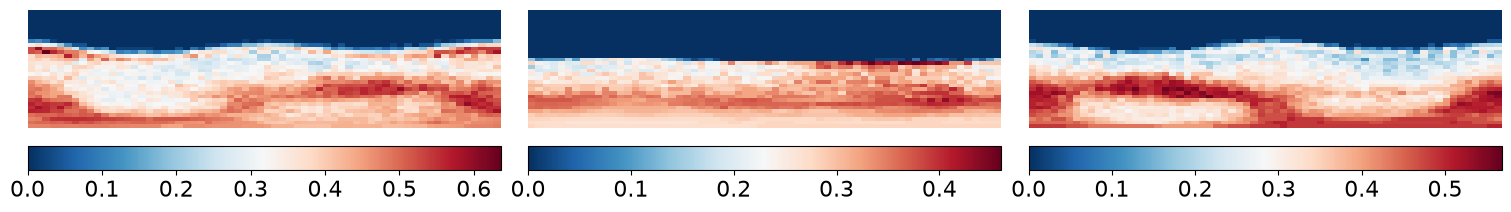}
    \caption{Sinogram corresponding to $f^\dagger$. Rows 1, 2 and 3 correspond to 3D versions Experiments~\ref{experiment: stereographic},~\ref{experiment: poincare} and \ref{experiment: focusing lens}, respectively. From left to right: slices at $(\theta, \varphi) \approx (\pi / 4, \pi / 2)$, $(\theta, \varphi) \approx (\pi / 2, \pi)$ and $(\theta, \varphi) \approx (3\pi / 4, 3\pi / 2)$.}
    \label{figure: slices of sinograms spatial}
\end{figure}

\begin{Experiment}\label{experiment: stereographic}
    We first consider the spherical stereographic metric, $\lambda = \lambda_1$ defined as in \eqref{equation: lambda functions}. The spherical metric gives the smallest reconstruction errors in both dimensions. SGD and CGM achieve similar accuracy, but CGM reaches the prescribed tolerance in substantially fewer iterations. The reconstructions and pointwise error maps are shown in the first row of Figure~\ref{figure: reconstructions and errors} and in Figure~\ref{figure: reconstructions and errors spatial spherical_2}.
\end{Experiment}

\begin{Experiment}\label{experiment: poincare}
    In the second experiment, we use the Poincar\'e metric, $\lambda = \lambda_2$ as in \eqref{equation: lambda functions}. This experiment requires the most iterations in both dimensions. In 2D, both methods leave a relative error of about $8.3\%$, despite the large difference in the number of iterations. The reconstructions in the second row of Figure~\ref{figure: reconstructions and errors} and in Figure~\ref{figure: reconstructions and errors spatial poincare_1_1} show ring-like artifacts near the center, suggesting a limitation of the measurement sampling. In the Poincar\'e geometry, only geodesics entering $M$ close to the inward normal reach this region; in 2D setting see the sample geodesics in the second row of Figure~\ref{figure: geodesics and sinograms}. Therefore, a uniform Cartesian grid of $[0, 2\pi)\times (-\pi/2, \pi/2)$, which is a parameterization of $\p_+ SM$, leaves the region near the center sparsely sampled. Similar observations have been made previously in \cite[Section~3.2]{Mo2014SIAMJIS}. The artifacts are related to sampling the geodesics which was studied in \cite{MoSt2023SIAMJMA, St2023IP}.
\end{Experiment}

\begin{Experiment}\label{experiment: focusing lens}
    Our final experiment concerns the focusing lens metric, $\lambda = \lambda_3$ given in \eqref{equation: lambda functions}. This experiment lies outside the geometric assumptions of the reconstruction algorihms of Section~\ref{section: reconstructions} and is included to examine the numerical behavior of the reconstruction methods in this setting. Despite the presence of conjugate points, the planar focusing-lens reconstruction is more accurate than the planar Poincaré reconstruction. The ordering changes in 3D, where the focusing lens gives the largest errors among the three metrics; see
Figure~\ref{figure: reconstructions and errors spatial focusing_lens_1_2}. CGM nevertheless improves on SGD in both dimensions. The 2D reconstructions in the third row of Figure~\ref{figure: reconstructions and errors} exhibit short streaks radiating from a region near $(0.2,0)$. This pattern is consistent with geodesics refocusing near a conjugate point, with reconstruction errors forming streaks that radiate outward from the surrounding neighborhood; see the sample geodesics in the third row of Figure~\ref{figure: geodesics and sinograms}. The connection between conjugate points and reconstruction artifacts was studied in \cite{MoStUh2015CMP}.
\end{Experiment}

\begin{figure}
    \centering
    \includegraphics[width=7.5cm]{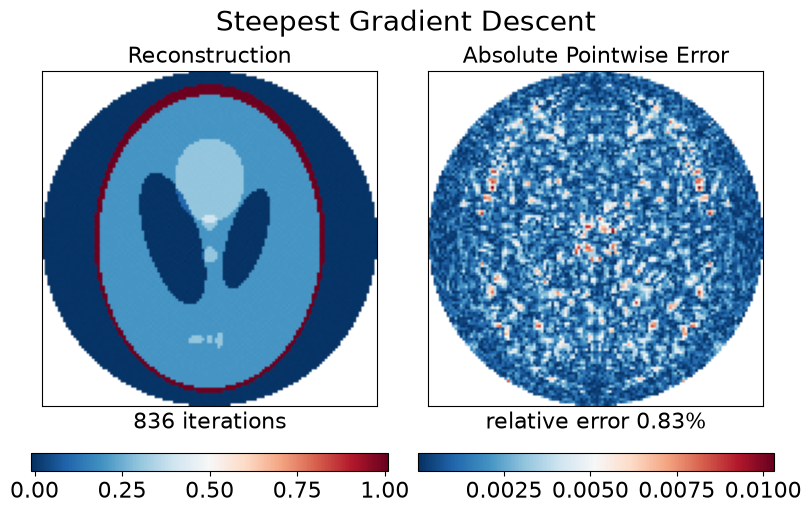}
    \includegraphics[width=7.5cm]{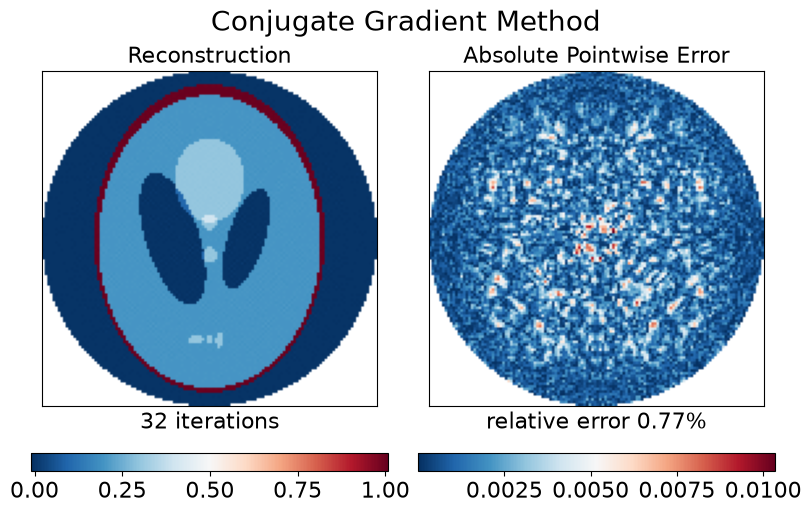}\\
    \includegraphics[width=7.5cm]{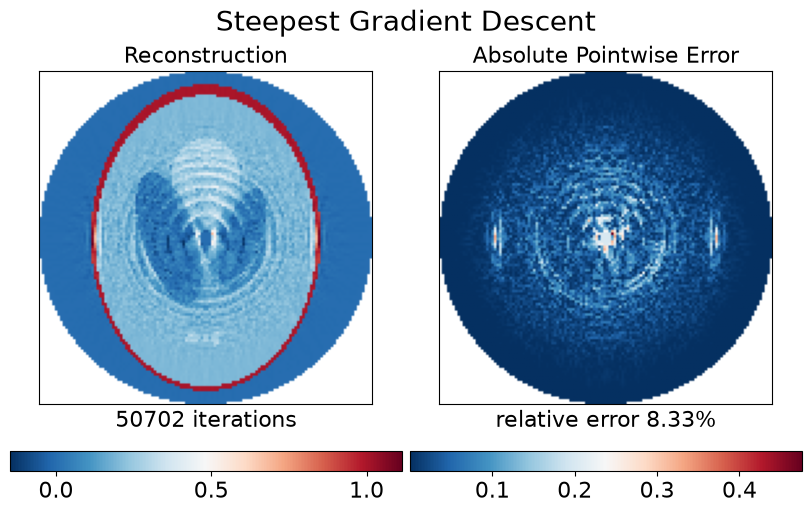}
    \includegraphics[width=7.5cm]{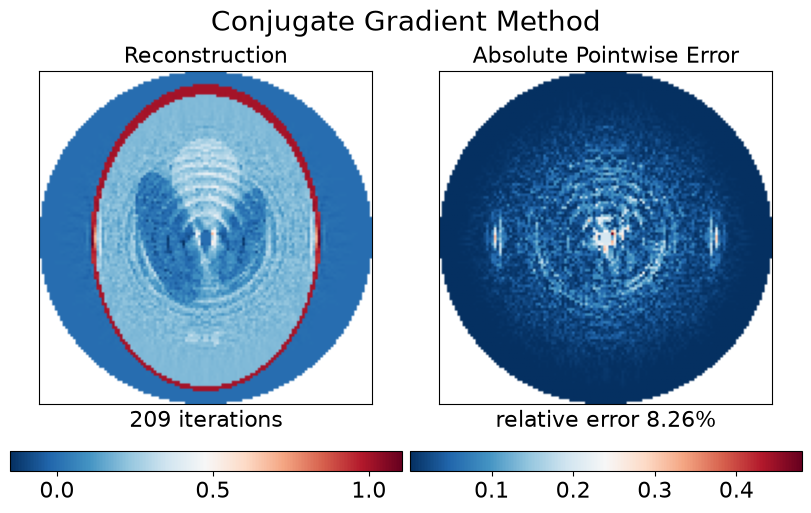}\\
    \includegraphics[width=7.5cm]{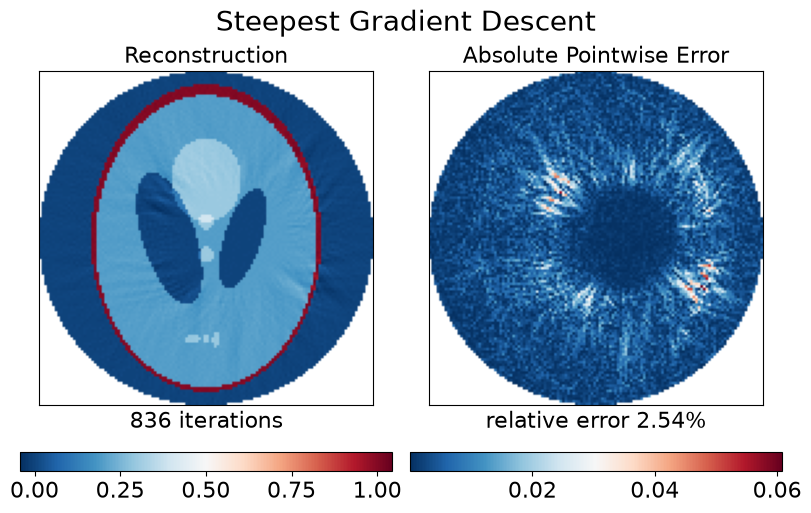}
    \includegraphics[width=7.5cm]{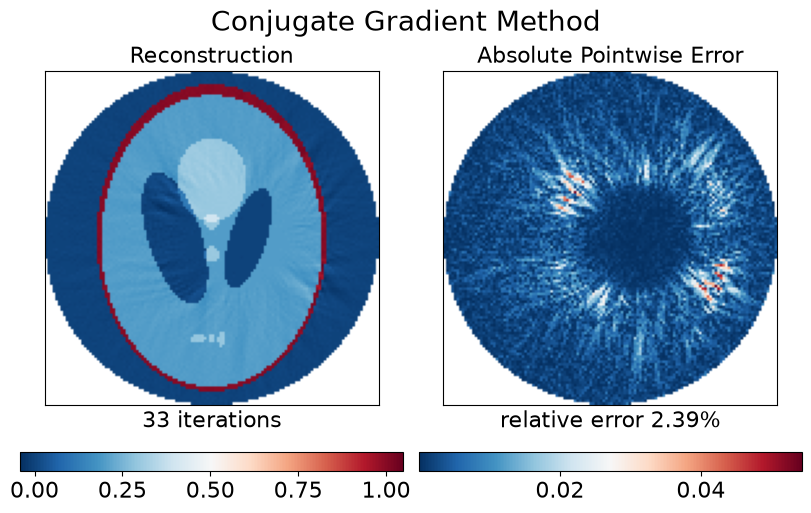}
    \caption{Reconstructions and absolute pointwise errors via SGD and CGM. Rows 1, 2 and 3 correspond to 2D versions of Experiments~\ref{experiment: stereographic},~\ref{experiment: poincare} and \ref{experiment: focusing lens}, respectively. The stopping rule is~\eqref{equation: stopping rule} with $\varepsilon = 0.001$.}
    \label{figure: reconstructions and errors}
\end{figure}

\begin{figure}
    \centering
    \includegraphics[width=11cm]{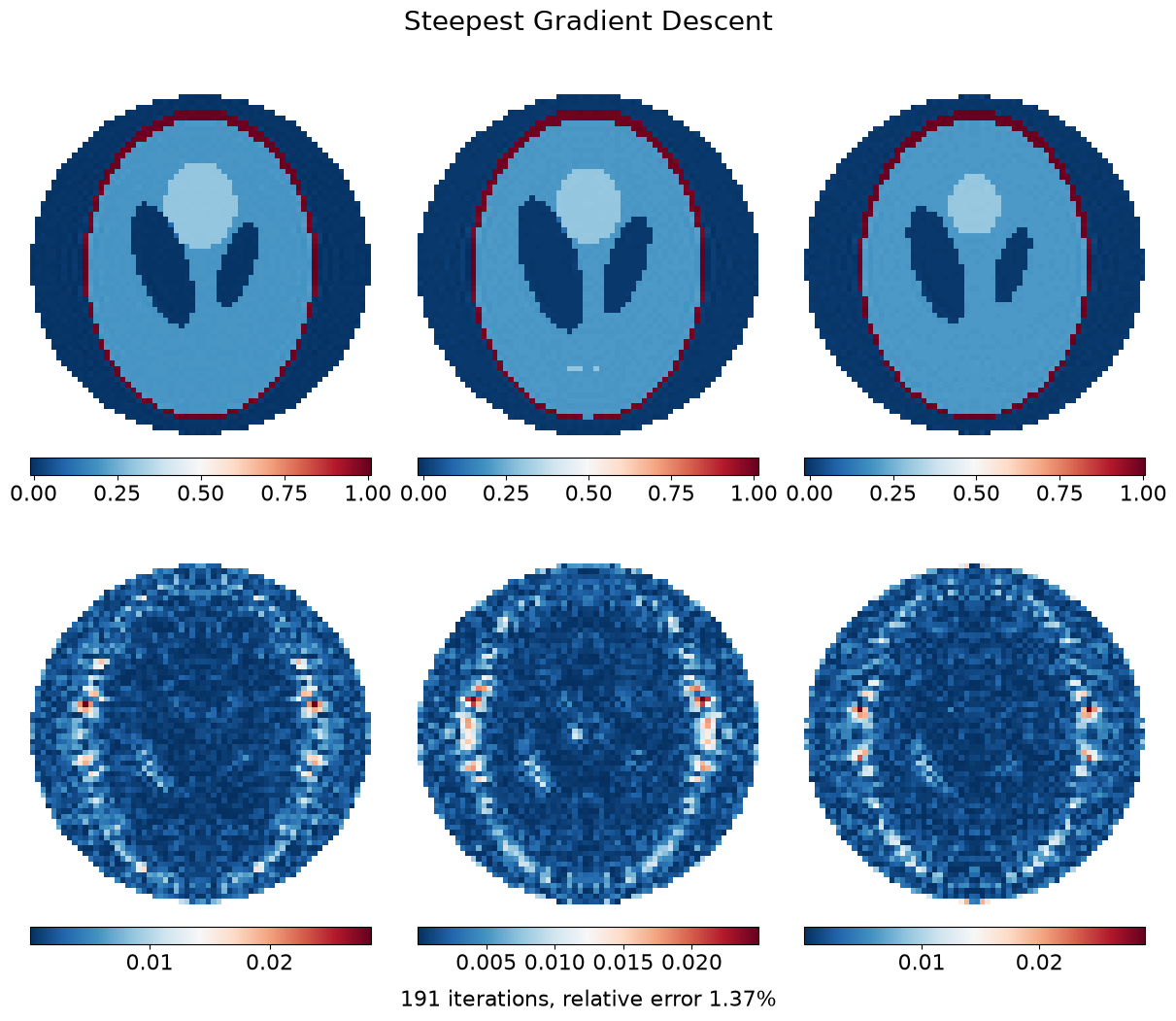}\\
    \includegraphics[width=11cm]{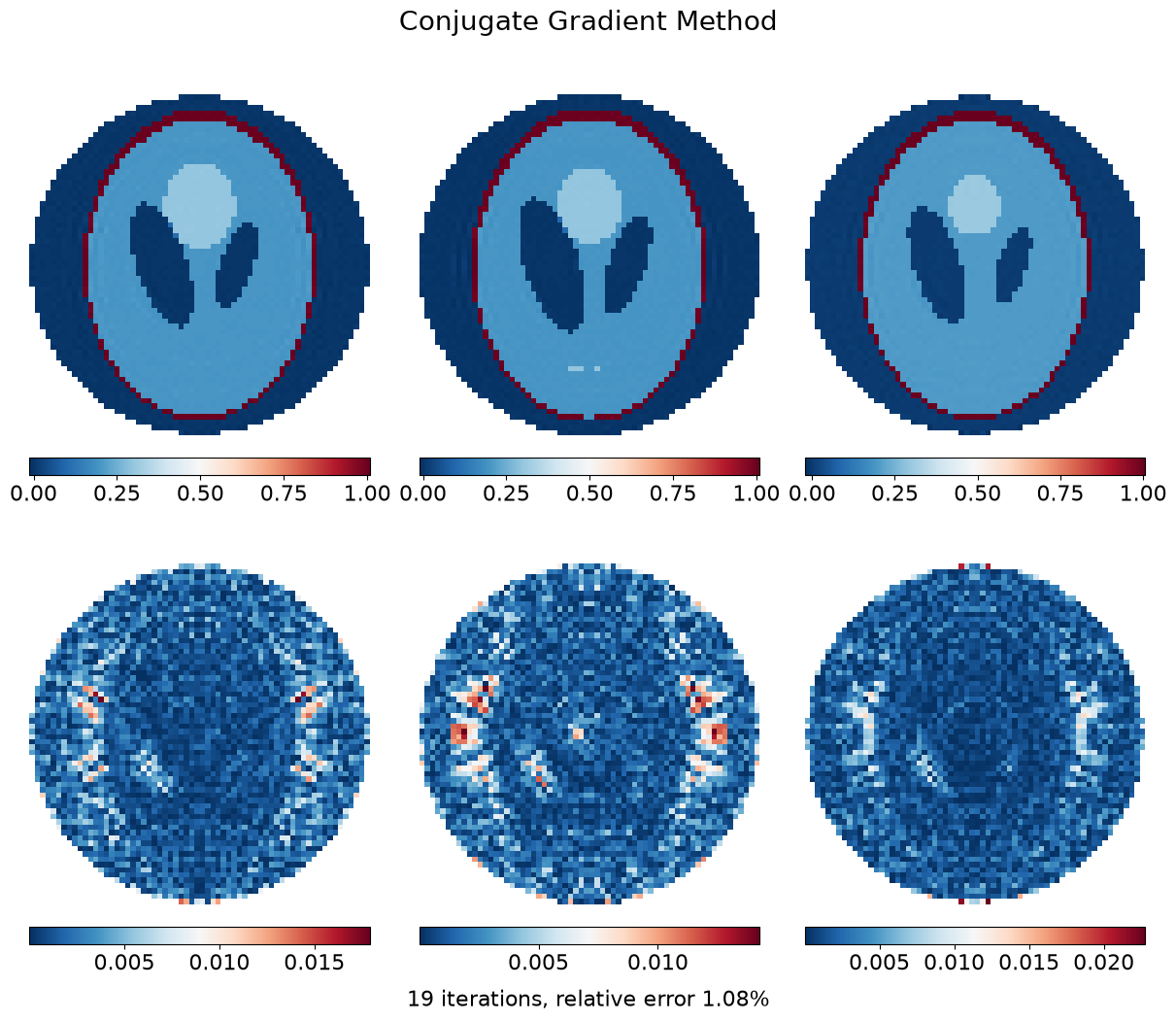}
    \caption{3D version of Experiment~\ref{experiment: stereographic}. Reconstructions and absolute pointwise errors via SGD and CGM. From left to right: slices at $z \approx -0.125$, $z \approx 0.0$ and $z \approx 0.125$. The stopping rule is~\eqref{equation: stopping rule} with $\varepsilon = 0.01$.}
    \label{figure: reconstructions and errors spatial spherical_2}
\end{figure}

\begin{figure}
    \centering
    \includegraphics[width=11cm]{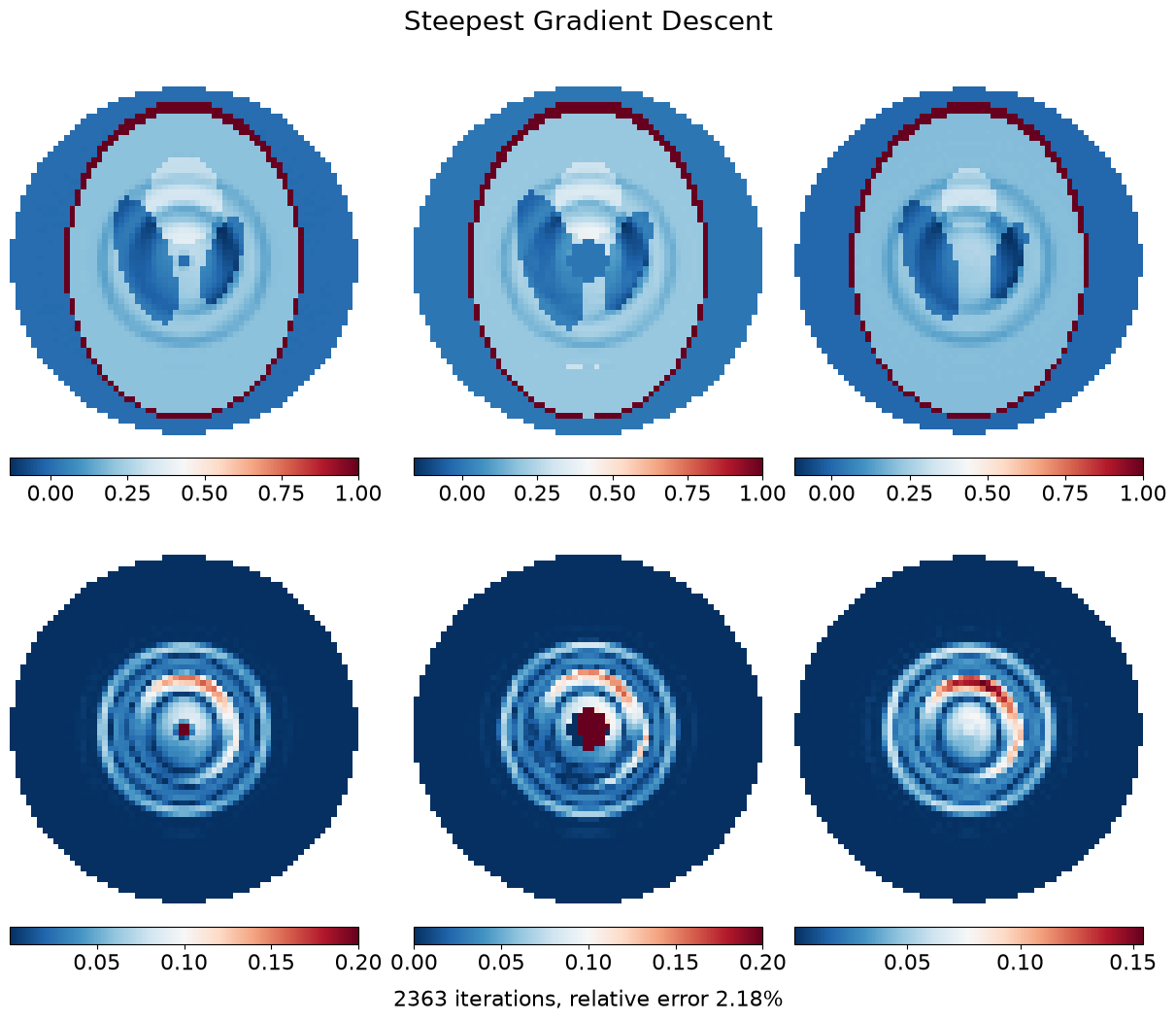}\\
    \includegraphics[width=11cm]{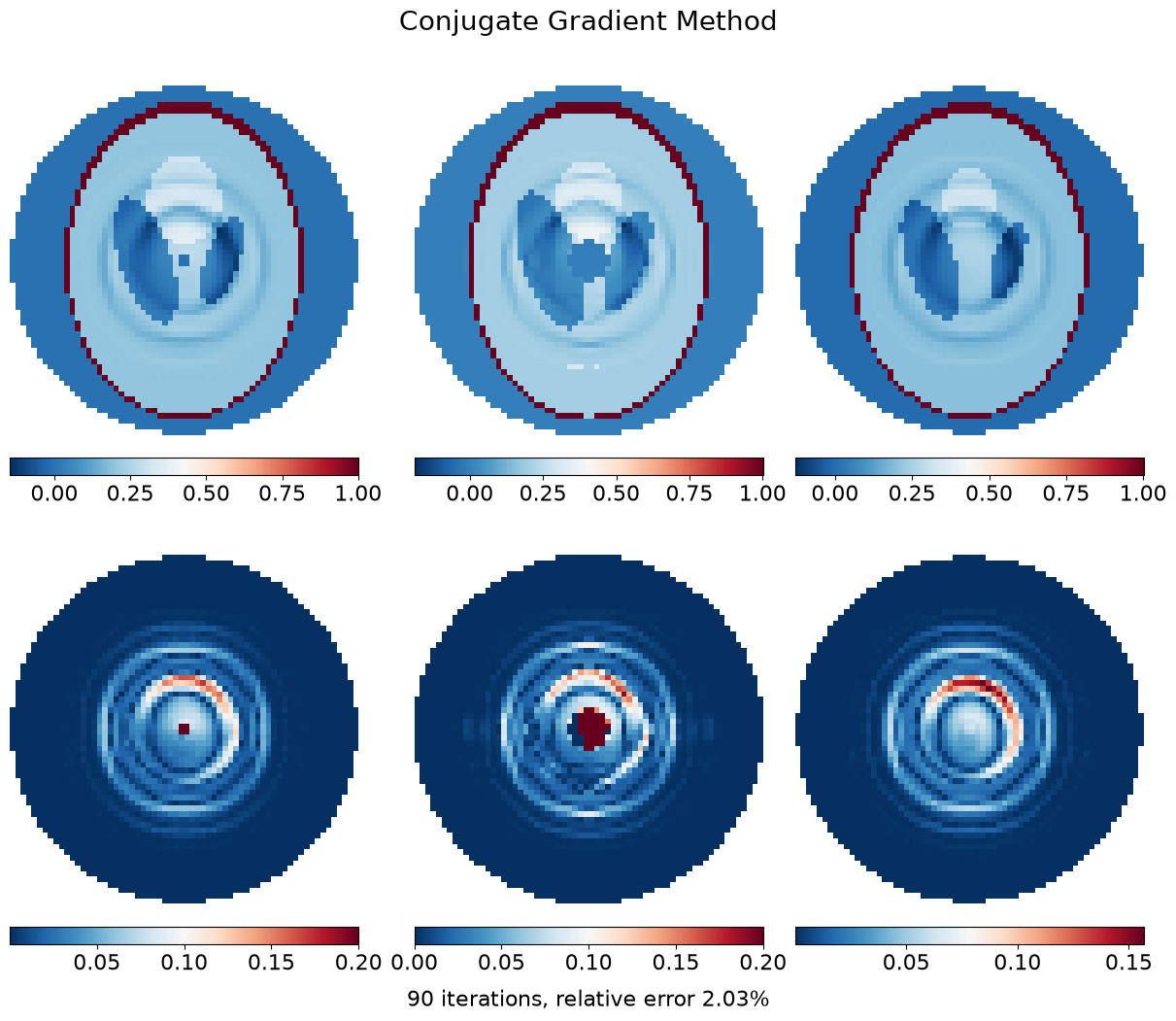}
    \caption{3D version of Experiment~\ref{experiment: poincare}. Reconstructions and absolute pointwise errors via SGD and CGM. From left to right: slices at $z \approx -0.125$, $z \approx 0.0$ and $z \approx 0.125$. The stopping rule is~\eqref{equation: stopping rule} with $\varepsilon = 0.01$.}
    \label{figure: reconstructions and errors spatial poincare_1_1}
\end{figure}

\begin{figure}
    \centering
    \includegraphics[width=11cm]{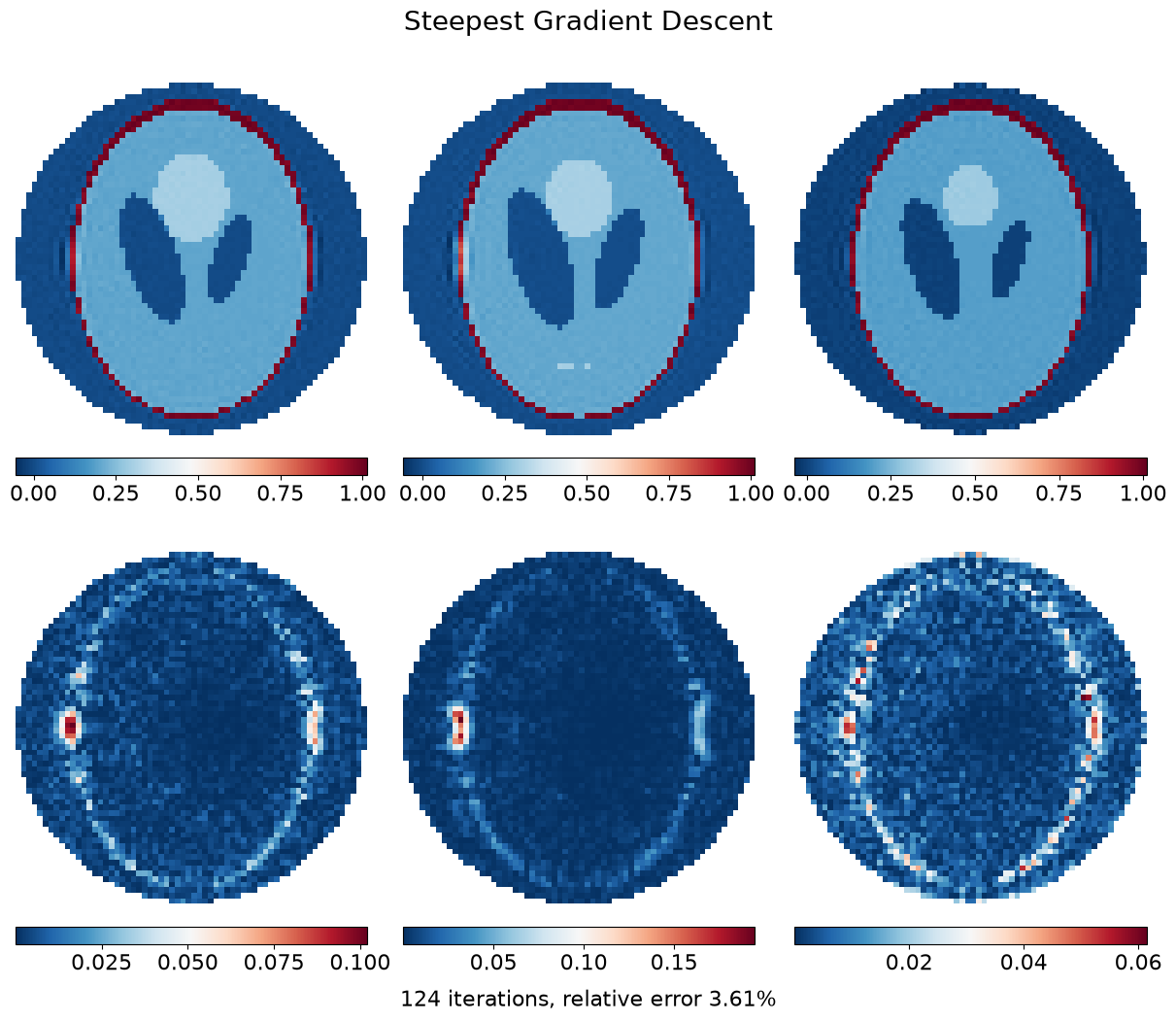}\\
    \includegraphics[width=11cm]{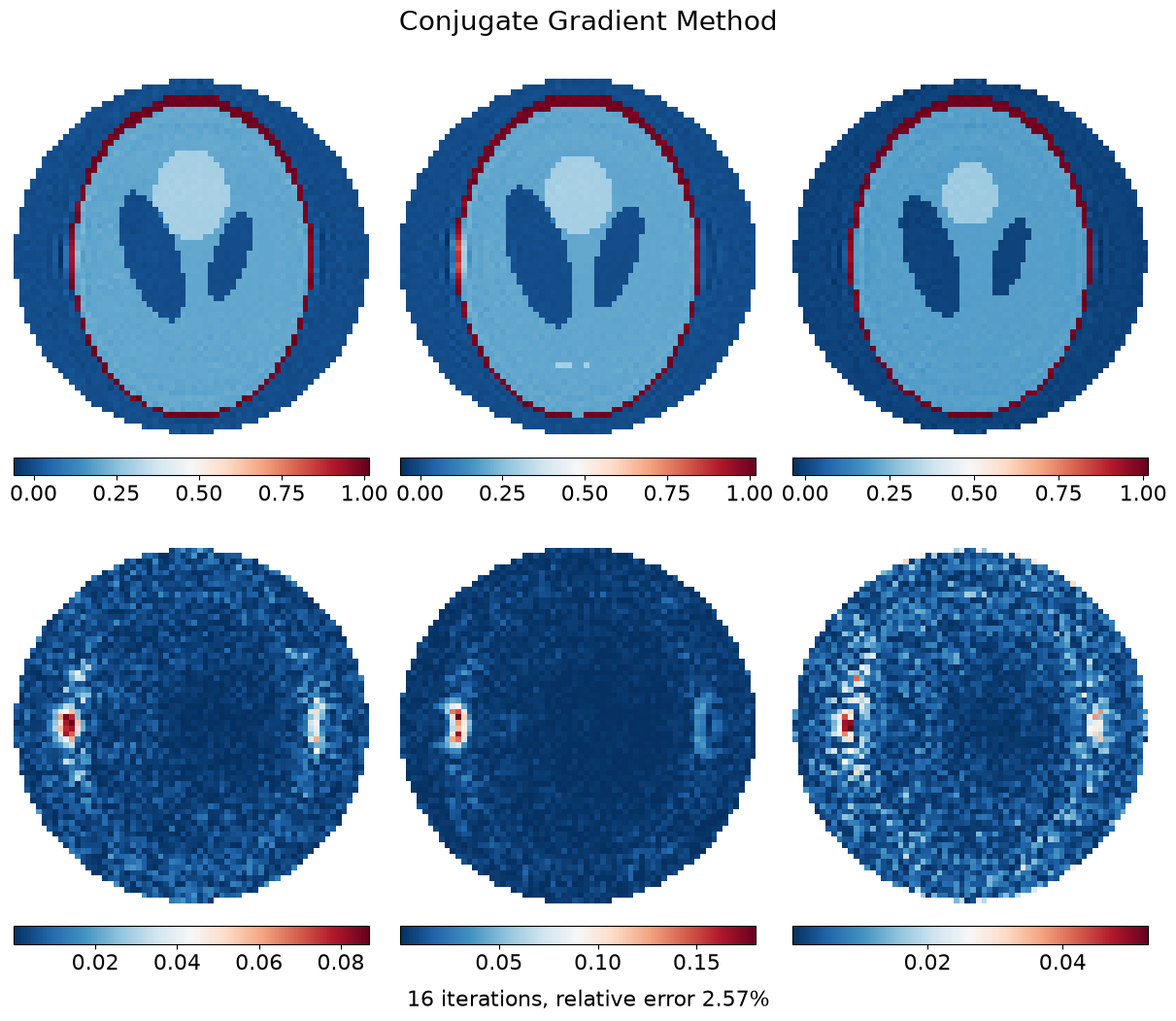}
    \caption{3D version of Experiment~\ref{experiment: focusing lens}. Reconstructions and absolute pointwise errors via SGD and CGM. From left to right: slices at $z \approx -0.125$, $z \approx 0.0$ and $z \approx 0.125$. The stopping rule is~\eqref{equation: stopping rule} with $\varepsilon = 0.01$.}
    \label{figure: reconstructions and errors spatial focusing_lens_1_2}
\end{figure}

\section{Conclusion}\label{section: conclusion}
We have studied the inversion of the geodesic ray transform from finitely many local averages of its data. Under the geometric assumptions of Section~\ref{section: stability estimates}, sufficiently fine measurement partitions preserve Lipschitz stability on any fixed finite-dimensional reconstruction space $\cW\subset L^2(M)$. We have also developed iterative reconstruction algorithms and illustrated their performance in 2D and 3D. In all of the reported experiments, CGM requires substantially fewer iterations than SGD while achieving comparable or smaller reconstruction errors. The reconstructions obtained by SGD and CGM exhibit artifacts whose appearances depend on the geometry of $(M,g)$ and the discretization of $\p_+ SM$.

The remaining artifacts motivate our ongoing work on deep-learning-based reconstruction methods for the geodesic ray transform. We are pursuing three directions to be developed in forthcoming papers:
\begin{itemize}
    \item We investigate whether a convolutional neural network (CNN) applied to CGM reconstructions can reduce these artifacts. This hybrid approach adapts the reconstruction-and-postprocessing strategy of \cite{JiMcFrUn2017TIP}, in which the inversion of the Radon transform via the filtered backprojection algorithm is combined with a CNN. See also \cite{XiZhChXiLiZhYaZhHu2018SR} for a related study.
    
    \item Inspired by the learned iterative approach of \cite{AdOk2017IP}, we investigate neural network architectures obtained by unrolling SGD and CGM. In other words, learning is incorporated into the reconstruction steps rather than applied only to the final image.
    
    \item We will use operator learning to approximate
    $$
        (I_g\imath_\delta^M)^{-1}: I_g(V_\delta(M)) \longrightarrow V_\delta(M),
        \qquad I_g(V_\delta(M))\subset L^2_\mu(\p_+ SM).
    $$
    We aim to establish approximation guarantees for this end-to-end approach and assess its reconstruction accuracy and potential for reducing artifacts through numerical experiments in 2D and 3D.
\end{itemize}

\bibliographystyle{abbrv}

\end{document}